\documentclass[11pt]{amsart}

\usepackage{amssymb,bbm,amscd}
\usepackage{graphicx}
\usepackage{a4wide}

\theoremstyle{plain}
\newtheorem{theorem}{Theorem}
\newtheorem{prop}{Proposition}
\newtheorem{lemma}{Lemma}
\newtheorem{coro}{Corollary}

\theoremstyle{definition}
\newtheorem{remark}{Remark}

\newcommand{\dd}{\,\mathrm{d}}
\newcommand{\ii}{\ts\mathrm{i}}
\newcommand{\ee}{\,\mathrm{e}}
\newcommand{\ts}{\hspace{0.5pt}}

\DeclareMathOperator*{\bigdotcup}{\dot{\bigcup}}
\DeclareMathOperator{\dotcup}{\dot{\cup}}

\newcommand{\vL}{\varLambda}
\newcommand{\ZZ}{\mathbb{Z}}
\newcommand{\RR}{\mathbb{R}\ts}
\newcommand{\CC}{\mathbb{C}\ts}
\newcommand{\NN}{\mathbb{N}}
\newcommand{\TT}{\mathbb{T}}
\newcommand{\XX}{\mathbb{X}}
\newcommand{\YY}{\mathbb{Y}}

\newcommand{\bs}[1]{\boldsymbol{#1}}
  
\begin{document}

\title[Squirals and beyond]{Squirals and beyond:\ Substitution tilings
 \\[2mm] with singular continuous spectrum}

\author{Michael Baake}
\address{Fakult\"{a}t f\"{u}r Mathematik, Universit\"{a}t Bielefeld,\newline
\hspace*{\parindent}Postfach 100131, 33501 Bielefeld, Germany}
\email{mbaake@math.uni-bielefeld.de}

\author{Uwe Grimm}
\address{Department of Mathematics and Statistics,
The Open University,\newline 
\hspace*{\parindent}Walton Hall, Milton Keynes MK7 6AA, United Kingdom}
\email{u.g.grimm@open.ac.uk}

\begin{abstract}
  The squiral inflation rule is equivalent to a bijective block
  substitution rule and leads to an interesting lattice dynamical
  system under the action of $\ZZ^{2}$.  In particular, its balanced
  version has purely singular continuous diffraction. The dynamical
  spectrum is of mixed type, with pure point and singular continuous
  components. We present a constructive proof that admits a
  generalisation to bijective block substitutions of trivial height
  on $\ZZ^{d}$.
\end{abstract}

\maketitle
\thispagestyle{empty}

\section{Introduction}

Dynamical systems of translation bounded measures on $\RR^{d}$ with
pure point spectrum are well understood by now. This owes a lot to the
equivalence between pure point dynamical and diffraction spectra
\cite{LMS,BL}, and to the general characterisation of model sets via
dynamical systems; see \cite{BLM} and references therein. From an
applied perspective, pure point spectra are linked to the analysis of
periodic and almost periodic systems, such as crystals and
quasicrystals, by standard crystallographic methods; see \cite{Cowley}
for background, and \cite{BGrev} for a recent review in the aperiodic
setting.  In fact, the understanding of pure point systems has
improved sufficiently that one can begin to attack the corresponding
inverse problem systematically \cite{LenMoo}.

As soon as one enters the realm of mixed spectra, the picture is less
transparent. While some partial understanding exists for systems with
absolutely continuous spectral components (for instance through the
connection with the highly developed field of stochastic processes;
see \cite{BBM} and references therein), one is pretty much in the dark
when it comes to systems with singular continuous spectra.  The
classic Thue-Morse system, see \cite{AS} for background, and its
generalisations in the spirit of \cite{Keane,BGG} are notable exceptions,
which are all one-dimensional.

Beyond the theoretical interest, there was little motivation in the
past to dive further into systems with continuous spectral components,
at least not from an applied point of view. However, with the modern
measurements possible in materials science, spectra of this type are
detectable and observed more frequently \cite{Withers}, so that some
further analysis is needed. One obstacle has been that practically no
examples in higher dimensions are known, and certainly not in an
explicit fashion. This, in turn, is where the classic Thue-Morse
system excels:\ Since its original (spectral) discussion in
\cite{Wie,Mah} and its later reformulation in \cite{Kaku}, it has been
the paradigm of singular continuous (dynamical) spectrum, and it is
actually rather easy to also calculate its distribution function
explicitly and with high precision; see \cite{BG08} and references
therein for details and \cite{Zaks} for an approach to the analysis of
its `fractal' aspects. The good accessibility of the latter is a
consequence of the underlying Riesz product structure; see \cite{Zyg}
for background material.

However, the situation is not as bad as it appears. A substantial step
forward was achieved in \cite{LM,LMS,LMS2} for lattice substitution
systems by identifying modular coincidence as one viable and powerful
generalisation of Dekking's criterion \cite{Dekking} for pure
pointedness in one-dimensional substitutions of constant length. The
interesting cases are the systems that \emph{fail} to possess a
modular coincidence, such as the bijective substitutions of constant
length and their higher-dimensional generalisations studied in
\cite{Nat1,Nat2}. However, to the best of our knowledge, no serious
attempt has been made to investigate a genuinely higher-dimensional
example with singular continuous spectrum explicitly (by which we mean
an example that cannot be written as a product of one-dimensional
systems, such as those appearing in \cite{HB}).

Here, we report on a planar example that was identified as a good
candidate in \cite{Dirk}, and could be tackled with the methods
explained in \cite{FS}. It it known as the \emph{squiral tiling} and
appears in \cite[Fig.~10.1.4]{GS}, where it was constructed as a
simple example for a tiling of the plane by one prototile with
infinitely many edges (and its mirror image). It is obtained by means
of the simple inflation rule
\begin{equation}\label{eq:sqinfl}
   \raisebox{-8ex}{\includegraphics[width=0.4\textwidth]{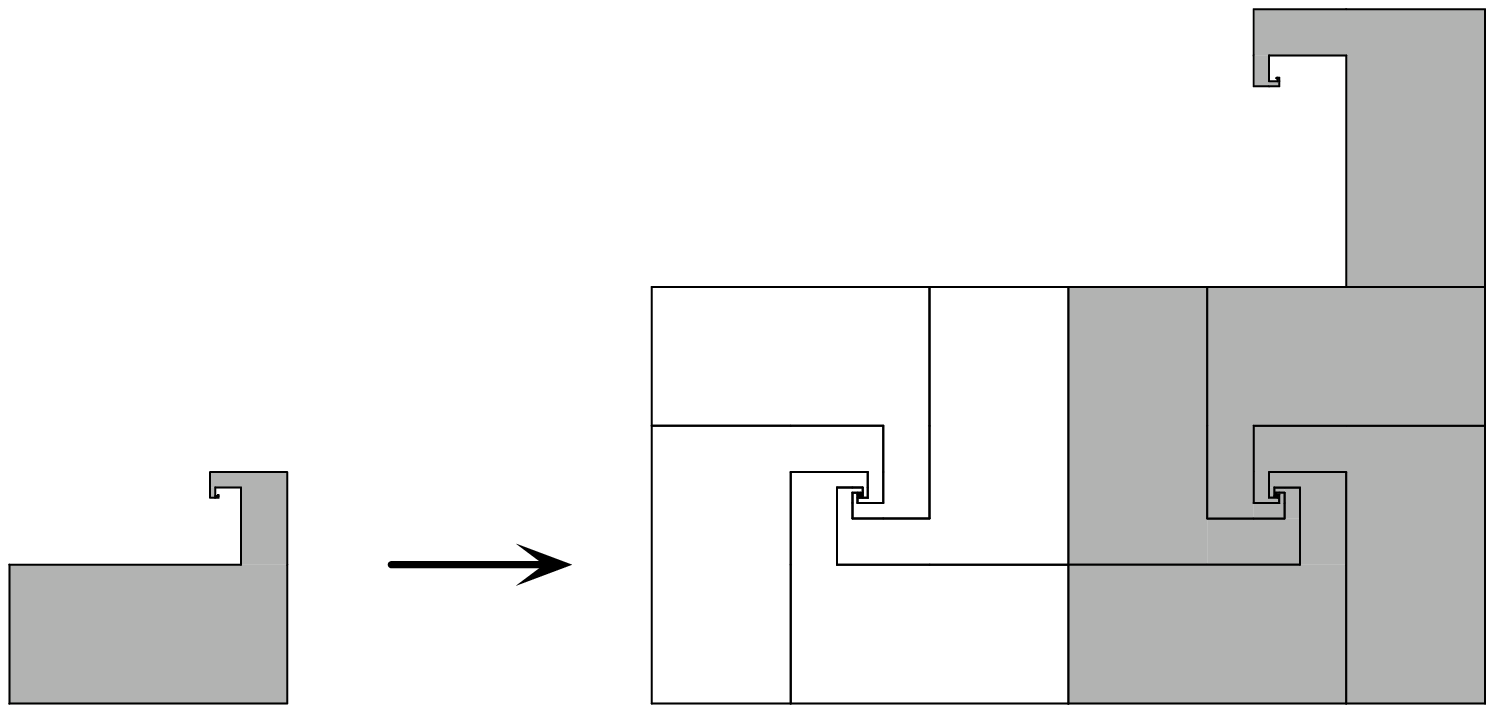}}
\end{equation}
together with the corresponding rule for the opposite chirality of
the squiral prototile. It has the integer inflation multiplier $3$
and defines an aperiodic tiling of the Euclidean plane. The term
`aperiodicity' is used in its strong version here, hence meaning
that the inflation rule defines a unique hull (via the closure
of the $\ZZ^{2}$-orbit of a fixed point in the local topology)
with the property that no element of this hull admits a non-trivial
period. A larger patch of the tiling is shown in Figure~\ref{fig:sqtil}.

\begin{figure}
\begin{center}
  \includegraphics[width=0.9\textwidth]{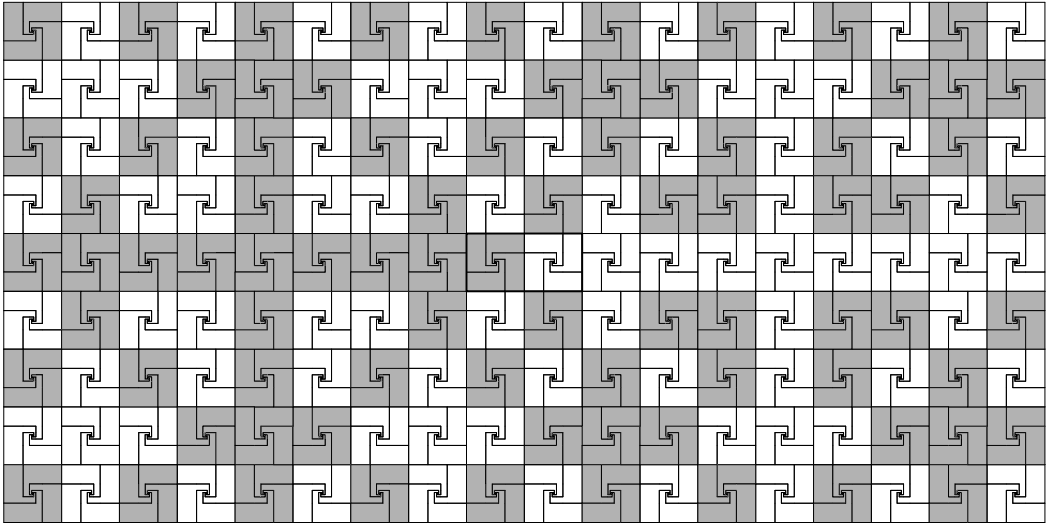}
\end{center}
\caption{\label{fig:sqtil} A rectangular patch of the squiral tiling
  with a `black and white' reflection symmetry in the vertical
  axis. The displayed patch was obtained via two inflation steps of
  \eqref{eq:sqinfl} from the central rectangular seed (marked), the
  latter comprising four tiles of each chirality. }
\end{figure}

As an application of \cite{FS}, Frettl\"{o}h and Sing checked that
this tiling has no modular coincidence, so cannot be a model set,
wherefore this is a candidate for a mixed spectrum. In what follows,
we reformulate this tiling via a topologically conjugate one (by
mutual local derivability \cite{BSJ,B02}) and bring it to the simpler
setting of block or lattice substitutions.  It will then be possible
to show that its balanced version (with weights $1$ and $\bar{1}=-1$
of equal frequency) has purely singular continuous diffraction
spectrum. As a consequence, the dynamical spectrum will be a mixture
of a pure point and a singular continuous part, in line with the
(implicit) conjecture in \cite{Nat1}.

Our approach (as indicated already) will be constructive, so that we
do not only determine the spectral type, but also derive the
diffraction measure explicitly. In fact, we can identify it as a
two-dimensional Riesz product. This means that one can formulate a
simple recursion for a sequence of 2D continuous distribution
functions that converge towards the distribution function of the
squiral measure, and uniformly so (even though the underlying measures
are absolutely continuous, and can thus only converge to the squiral
measure in the vague topology).

In Section~\ref{sec:gen}, this approach is generalised to primitive
and bijective block substitutions of constant length over a binary
alphabet, in arbitrary dimension $d$. We prove that these systems show
singular continuous diffraction if the height lattice is trivial.

\section{Squirals and lattice inflations} 

The squiral inflation of Eq.~\eqref{eq:sqinfl} is primitive, with
inflation matrix $M=\left(\begin{smallmatrix} 5 & 4 \\ 4 &
    5 \end{smallmatrix}\right)$.  The latter has eigenvalues $9=3^2$
and $1$, with Perron-Frobenius eigenvector $(1,1)$, which is in line
with both chiralities occupying the same area and occurring with equal
frequency (via reading it as left or as right eigenvector). Moreover,
if we add one pseudo-vertex at the middle of the long edge of the
squiral, the inflation tiling is face to face as well.  As such, it
defines a strictly ergodic dynamical system under the shift action of
$\ZZ^{2}$. Note that the special (spiralling) vertex always falls on
the centre of a square that is formed by $4$ squirals of the same
chirality. This also explains the name as a mixture of `square' and
`spiral'.  It is natural to view the squiral tilings as a (local)
decoration of the square lattice, where the colour simply codes the
chirality of the square decoration. This correspondence gives rise to
a mutually local rule in the sense of \cite{BSJ,B02} as follows.

The first step consists of identifying an induced inflation rule on
the coloured square lattice, which is simply given by
\begin{equation}\label{eq:blocksub}
  \raisebox{-6ex}{\includegraphics[width = 0.6 \textwidth]{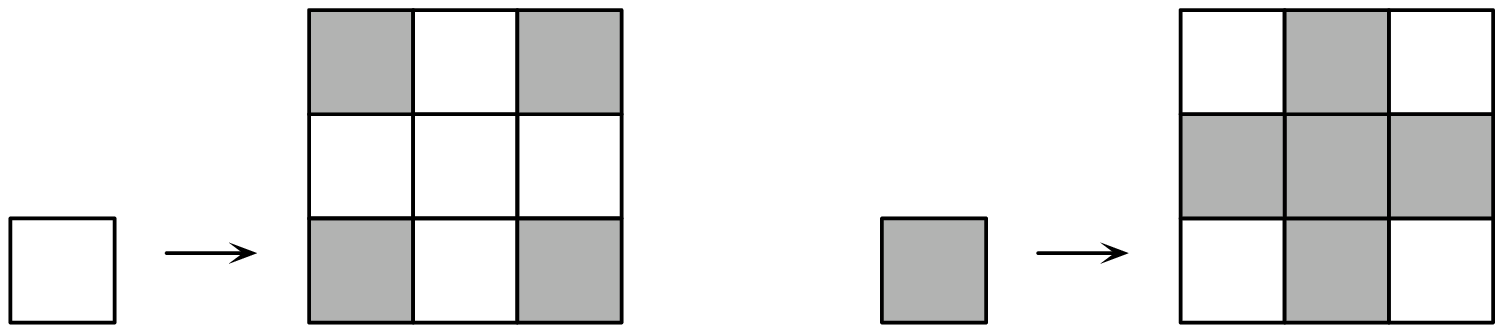}}
\end{equation}
This defines a block substitution that is bijective and of constant
length in the terminology of \cite{Nat1}. As a coloured tiling, it is
thus defined by a primitive inflation rule (with the same inflation
matrix $M$ as before). One quickly checks that the inflation action
and the derivation rule (now applicable in either direction) commute,
so that the block substitution defines a hull that is mutually locally
derivable (MLD) to the squiral hull.  It is clear that the former is
better suited for our further analysis.

\begin{figure}
\begin{center}
  \includegraphics[width=0.9\textwidth]{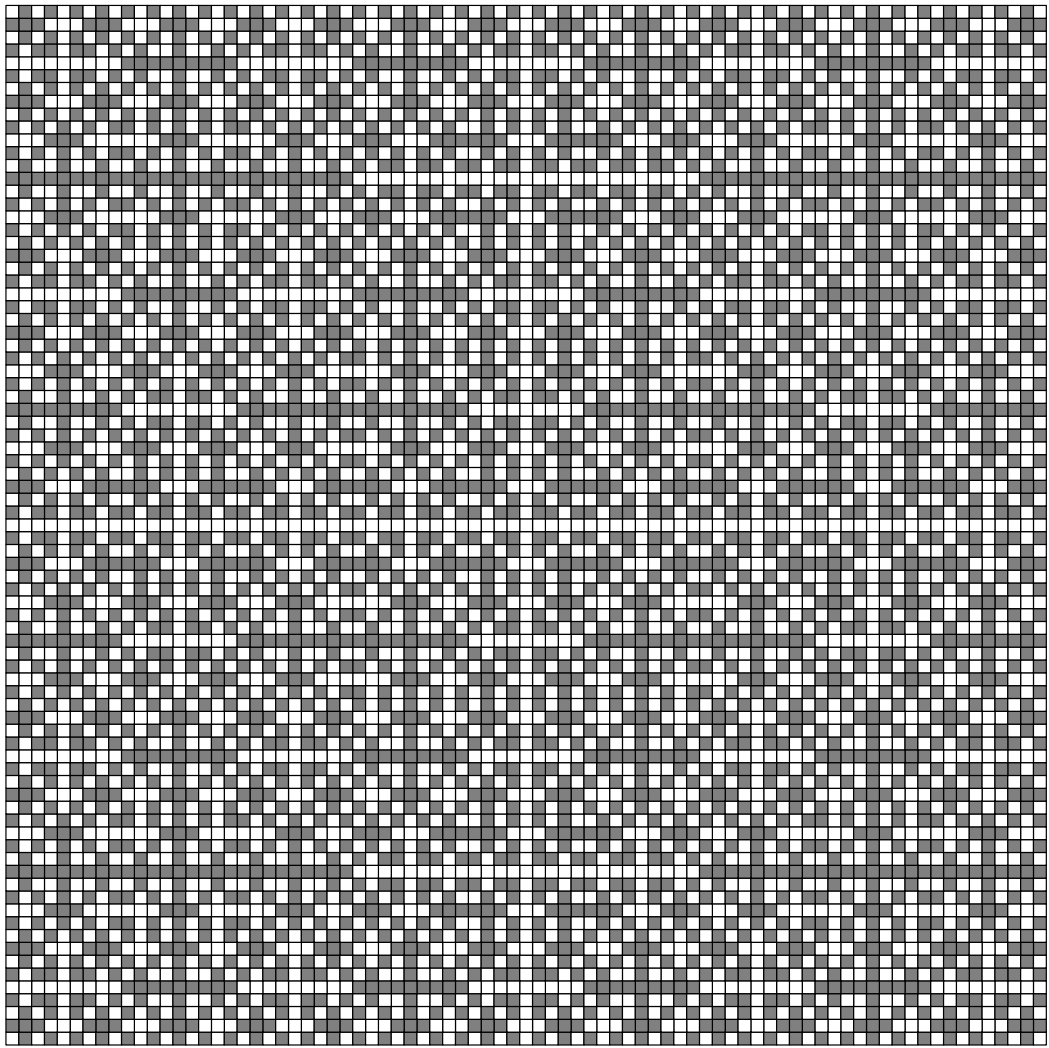}
\end{center}
\caption{\label{fig:blocktil} A square-shaped patch of the squiral
 tiling with full $D_{4}$-symmetry, as obtained from the equivalent
 block substitution \eqref{eq:blocksub}.}
\end{figure}

If one iterates the block substitution starting from a single square
with its centre as reference point, one obtains a sequence of growing
(finite) tilings that converges to a tiling of the plane which is a
fixed point; see Figure~\ref{fig:blocktil} for an illustration. This
particular tiling has perfect $D_{4}$-symmetry and defines (via its
hull, the closure of its $\ZZ^{2}$-orbit in the local topology) a
minimal dynamical system that is uniquely ergodic. This follows from
standard arguments because the inflation is primitive and leads to a
face to face tiling of finite local complexity.  As a consequence, the
fixed point is linearly repetitive, which implies the claim by
\cite[Thm.~6.1]{LP}. For some general results on tiling dynamical
systems, we refer to \cite{Rob,Sol,Nat1}.

\section{Autocorrelation coefficients}

Let us now represent each (coloured) square by a point at its lower
left corner that carries a matching colour, which we call $1$ (for
white) or $\bar{1}$ (for grey). The hull $\XX$ is then a closed
subspace of $\{ 1,\bar{1}\}^{\ZZ^{2}}$ that is invariant under the
shift action of $\ZZ^{2}$ by construction. Given $w\in\XX$, we define
the corresponding Dirac comb 
\begin{equation}\label{eq:def-comb}
  \omega \, = \, w \ts \delta^{}_{\ZZ^{2}} \, = 
  \sum_{z\in\ZZ^{2}} w^{}_{z} \ts \delta^{}_{z} \ts ,
\end{equation} 
where we write $z\in\ZZ^{2}$ as a pair of integers from now on. For
the weights, we use the convention $\bar{1}=-1$.  Following the
approach pioneered by Hof \cite{Hof}, we define the natural
autocorrelation measure $\gamma$ of $\omega$ as
\begin{equation}\label{eq:gam-def}
   \gamma \, = \, \omega \circledast \widetilde{\omega} \, := \,
   \lim_{N\to\infty} \frac{\omega\big|^{}_{C_{N}} \! * 
   \, \widetilde{\omega\big|^{}_{C_{N}}}}{(2N+1)^{2}}     
\end{equation}
where $C_{N}$ denotes the closed centred square of side length $2N$.
Here, $\widetilde{\mu}$ denotes the measure defined by
$\widetilde{\mu} (g) = \overline{\mu (\widetilde{g})}$ for $g\in
C_{\mathsf{c}} (\RR^{2})$, with $\widetilde{g} (x) :=
\overline{g(x)}$.  The limit in Eq.~\eqref{eq:gam-def} always exists
due to unique ergodicity, and is independent of the choice of $w \in
\XX$. The measure $\gamma$ is of the form $\gamma = \eta\ts\ts
\delta^{}_{\ZZ^{2}}$ with coefficients
\[
   \eta\ts (m,n) \, = \, \lim_{N\to\infty} \frac{1}{(2N+1)^{2}}
   \sum_{k,\ell=-N}^{N} w^{}_{k,\ell}\, w^{}_{k-m,\ell-n} \ts ,
\]
where the result is again independent of the choice of $w$. We
discuss the convergence of these sums in more detail later.

Due to the symmetry of our system, it suffices to consider the
positive quadrant and formulate the autocorrelation accordingly. This
simplifies our further calculations. To this end, we consider the
square inflation with the lower left corner as new reference point,
and observe that this, when starting from a single block, leads an
iteration sequence which fills the positive quadrant only.  To obtain
a complete tiling, one may start from the legal seed
$\begin{smallmatrix} \bar{1} & 1 \\ 1 & \bar{1} \end{smallmatrix}$
with reference point in its centre.  Recall that a patch is called
\emph{legal} when it occurs in the $n$-fold substitution of a single
letter, for some $n\in\NN$.  The iteration of the substitution then
converges towards a $2$-cycle that covers $\ZZ^{2}$, each element of
which defines the same hull. The substitution is still a $2$-cycle
when restricted to the positive quadrant. Due to the $D_{4}$-symmetry
of the pattern in Figure~\ref{fig:blocktil}, and hence that of the
entire hull $\XX$, it is sufficient to calculate the autocorrelation
coefficients within the positive quadrant.

Denote the two configurations in the positive quadrant (written in the
alphabet $\{1,\bar{1}\}$, with $\bar{\bar{1}}=1$) by $v$ and $\varrho
v$. They satisfy
\begin{equation}\label{eq:fix}
   (\varrho v)^{}_{3m+r,3n+s} \, = \, \begin{cases}
   \overline{v}^{}_{m,n}  , & 
       \text{if } r\equiv s\equiv 0 \bmod 2 , \\
   v^{}_{m,n} , & \text{otherwise},  \end{cases}
\end{equation}
where $m,n \ge 0$ and $0 \le r,s \le 2$.  The autocorrelation
coefficients clearly satisfy
\begin{equation}\label{eq:eta-def}
  \eta\ts (m,n) \, =  \lim_{N\to\infty} \frac{1}{N^{2}}
  \sum_{k,\ell =0}^{N-1} (\varrho v)^{}_{k,\ell} \,
  (\varrho v)^{}_{k+m,\ell+n} \, = 
  \lim_{N\to\infty} \frac{1}{N^{2}}
  \sum_{k,\ell =0}^{N-1} v^{}_{k,\ell} \, v^{}_{k+m,\ell+n} 
\end{equation}
for $m,n \ge 0$. All limits exist due to unique ergodicity and the
fact that the sum is an orbit average of a continuous function
\cite{W}. Moreover, we have used the symmetry to write $\eta$ via the
positive quadrant only. Clearly, we have $\eta\ts (0,0) = 1$ together
with
\begin{equation}\label{eq:eta-sym}
   \eta\ts (-m,n) \, = \, \eta\ts (m,-n) \, = \,
   \eta\ts (-m,-n) \, = \, \eta\ts (m,n) \ts ,
\end{equation}
which specifies the function $\eta$ on all of $\ZZ^{2}$.
Note that the $D_{4}$-symmetry of our system also
implies the relation $\eta\ts (m,n) = \eta\ts (n,m)$.

For our further analysis, we introduce the shorthand
\[
     \langle m,n\rangle  \, := \, \eta\ts (m,n)\ts .
\]
Considering the coefficient $\langle 3m \! + \! r,3n \! + \! s\rangle
$ for $\varrho v$ and fixed $m,n\ge 0$ and $0\le r,s \le 2$, one can
split the defining sum modulo $3$ and use Eqs.~\eqref{eq:fix} and
\eqref{eq:eta-def} to derive the recursion relations
\begin{equation}\label{eq:eta-rec}
  \begin{split}
  \langle 3m,3n\rangle  \,   & = \, \langle m,n\rangle  \\
  \langle 3m,3n \! + \!1\rangle  \, & = \, 
   - \tfrac{2}{9} \langle m,n\rangle  + 
             \tfrac{1}{3} \langle m,n \! + \! 1\rangle  \\
  \langle 3m,3n \! + \! 2\rangle  \, & = \,   
   \tfrac{1}{3} \langle m,n\rangle  - 
             \tfrac{2}{9} \langle m,n \! + \! 1\rangle  \\[1mm]
  \langle 3m \! + \! 1,3n\rangle  \,   & = \, 
   - \tfrac{2}{9} \langle m,n\rangle  + 
             \tfrac{1}{3} \langle m \! + \! 1,n\rangle  \\
  \langle 3m \! + \! 1,3n \! + \! 1\rangle  \, & = \, - \tfrac{2}{9} 
        \bigl( \langle m \! + \! 1,n\rangle  + 
               \langle m,n \! + \! 1\rangle \bigr)
               + \tfrac{1}{9} \langle m \! + \! 1,n \! + \! 1\rangle  \\
  \langle 3m \! + \! 1,3n \! + \! 2\rangle  \, & = \, - \tfrac{2}{9} 
        \bigl(\langle m,n\rangle  + 
        \langle m \! + \! 1,n \! + \! 1\rangle \bigr)
                        + \tfrac{1}{9} \langle m \! + \! 1,n\rangle  \\[1mm]
  \langle 3m \! + \! 2,3n\rangle    \, & = \,   \tfrac{1}{3} 
        \langle m,n\rangle  - \tfrac{2}{9} \langle m \! + \! 1,n\rangle  \\
  \langle 3m \! + \! 2,3n \! + \! 1\rangle  \, & = \, - \tfrac{2}{9} 
        \bigl(\langle m,n\rangle  + 
        \langle m \! + \! 1,n \! + \! 1\rangle \bigr)
                        + \tfrac{1}{9} \langle m,n \! + \! 1\rangle  \\
  \langle 3m \! + \! 2,3n \! + \! 2\rangle  \, & = \,   \tfrac{1}{9} 
       \langle m,n\rangle  - \tfrac{2}{9} 
                          \bigl(\langle m \! + \! 1,n\rangle  + 
       \langle m,n \! + \! 1\rangle \bigr) .
  \end{split}
\end{equation}
All relations are linear and of the form 
\begin{equation}\label{eq:eta-rec-2}
   \langle 3m \! + \! r,3n \! + \! s\rangle  \, = 
   \sum_{k = 0}^{\min(1,r)} \sum_{\ell=0}^{\min(1,s)}
   \alpha^{(r,s)}_{k,\ell}\ts 
   \langle m \! + \! k,n \! + \! \ell\rangle  \ts ,
\end{equation}
where the coefficients are elements of $\{-\frac{2}{9}, 0,
\frac{1}{9}, \frac{1}{3}, 1 \}$.

\begin{lemma}\label{lem:rec}
  The autocorrelation coefficients\/ $\langle m,n\rangle $ of the
  block substitution system exist for all\/ $m,n \in \ZZ$ and are
  completely determined by\/ $\langle 0,0\rangle =1$ together with the
  linear recursions \eqref{eq:eta-rec}, which hold for all\/ $m,n \in \ZZ$.

  In particular, one has the special values\/ $\langle 0,\pm 1\rangle
  =\langle \pm 1,0\rangle =-\frac{1}{3}$, $\langle \pm 1, \pm 1\rangle
  = \frac{1}{6}$, $\langle 0,\pm 2\rangle =\langle \pm 2,0\rangle
  =\frac{11}{27}$, $\langle \pm 2, \pm 2\rangle =\frac{7}{27}$ and\/
  $\langle \pm 1, \pm 2\rangle =\langle \pm 2, \pm 1\rangle =
  -\frac{8}{27}$. Moreover, the coefficients satisfy\/
  $(-1)^{m+n} \langle m,n \rangle > 0$ for all\/ $m,n \in \ZZ$.
\end{lemma}
\begin{proof}
  The existence of the coefficients is clear by unique ergodicity, as
  discussed earlier.  The recursion relations \eqref{eq:eta-rec}, for
  $m,n\ge 0$, follow from an elementary (though somewhat tedious)
  calculation as mentioned above. The initial condition $\langle
  0,0\rangle =1$ is clear, while the other special values for
  non-negative arguments can then be successively calculated from the
  recursions by solving linear equations. As is easily seen, the
  coefficients $\langle m,n \rangle$ with $m, n \ge 2$ are then
  determined recursively.

  One can now explicitly check that the same method (formally) also
  gives the other special values, which are in agreement with the
  symmetry relations \eqref{eq:eta-sym} and $\langle m,n\rangle
  =\langle n,m\rangle $. Indeed, one can verify that the recursions
  (extended to all $m,n\in\ZZ$) respect \emph{all}
  $D_{4}$-symmetries. The linear recursion relations thus determine
  $\langle m,n\rangle $ from $\langle 0,0\rangle $ on the entire
  lattice $\ZZ^{2}$.

  The positivity claim is certainly true for all $-2 \le m,n \le 2$.
  Inspecting the recursion relations \eqref{eq:eta-rec}, one sees that
  all terms on the right hand side of any single relation have the
  same sign, and that the parity on the left hand side matches the
  formula, so that our claim follows inductively. In particular, due
  to the recursive structure, no coefficient can vanish.
\end{proof}

Let us formulate another, rather surprising consequence of the
recursive structure, which will significantly simplify one of our
later estimates.

\begin{lemma}\label{lem:pos}
  One has\/ $\langle m,0\rangle^2 - \langle m,n\rangle^2 \ge 0$ for
  all\/ $m,n \in \ZZ$.
\end{lemma}
\begin{proof}
  Our claim follows if we show that $|\langle m,n \rangle|\le |\langle
  m,0 \rangle|$ holds for all $m,n\ge 0$.  Since $|\langle m,n
  \rangle|=(-1)^{m+n}\langle m,n \rangle$ by Lemma~\ref{lem:rec}, the
  absolute values satisfy the recursion relations \eqref{eq:eta-rec} with
  each coefficient on the right-hand sides replaced by its
  modulus. The inequalities are true for $0\le m,n\le 2$ (by
  inspection of the initial condition and the special values of
  Lemma~\ref{lem:rec}). The claim now follows by
  induction, where one applies the recursions once to differences of
  the form $|\langle 3m\!+\!r,0 \rangle| - |\langle 3m\!+\!r,3n\!+\!s
  \rangle|$, with $0\le r,s\le 2$, which are all non-negative.
\end{proof}

The autocorrelation coefficients define a positive definite function
on $\ZZ^{2}$. By the Herglotz-Bochner theorem, compare 
Lemma~\ref{lem:gen-Wiener} in the appendix, it is thus the
(inverse) Fourier transform of a unique positive measure on the
$2$-torus $\TT^{2} = [0,1)^{2}$, so that
\begin{equation}\label{eq:mu-def}
     \langle m,n\rangle  \, =  \int_{0}^{1} \int_{0}^{1} 
     \exp \bigl(2 \pi \ii (mx + ny)\bigr)
     \dd \mu (x,y) \, = \int_{0}^{1} \int_{0}^{1} 
     \cos \bigl( 2 \pi (mx + ny)\bigr)  \dd \mu (x,y) .
\end{equation}
The second equality follows from the $D_{4}$-symmetry of the
coefficients, which implies the corresponding symmetry for $\mu$.  As
$\langle 0,0 \rangle = 1$, we see that $\mu$ is a probability measure
on $\TT^{2}$.

\begin{remark}
The Herglotz-Bochner theorem can be used to construct examples of
positive definite functions on $\ZZ^{2}$ with $D_{4}$-symmetry that do
\emph{not} satisfy the inequalities of Lemma~\ref{lem:pos}. For
instance, $\alpha + (1-\alpha) (-1)^{m+n}$ is the Fourier transform of
$\alpha \ts \delta^{}_{(0,0)} + (1-\alpha)\ts \delta^{}_{(\frac{1}{2},
  \frac{1}{2})}$, but violates the inequalities for $\alpha\in (0,1)$
and $m+n$ odd.
\end{remark}

Before we continue our analysis of the planar case, let us look at
an important one-dimensional subsystem. It will reappear later in
an essential way. Moreover, it serves to introduce the methods we use.

\section{A rank $1$ subsystem}\label{sec:1D}

Consider $\epsilon (m) = \langle m,0\rangle $ for $m\in\ZZ$, which
defines a positive definite function on $\ZZ$, with $(-1)^{m} \epsilon
(m) > 0$ for all $m\in\ZZ$ by Lemma~\ref{lem:pos}. One has $\epsilon
(0)=1$ together with the recursions
\begin{equation}\label{eq:eps-rec}
  \begin{split}
  \epsilon(3m) \, &= \, \epsilon(m)\\
  \epsilon(3m+1) \, &= \, -\frac{2}{9}\epsilon(m) + 
                  \frac{1}{3}\epsilon(m+1)\\
  \epsilon(3m+2) \, &=\, \frac{1}{3}\epsilon(m)-
                  \frac{2}{9}\epsilon(m+1) \ts ,
  \end{split}
\end{equation}
which hold for all $m\in\ZZ$.  We consider the measure $\gamma =
\epsilon\ts \delta^{}_{\ZZ}$ on the lattice $\ZZ$, which corresponds
to the autocorrelation of our planar system along the $x$ direction
(or, by symmetry, along the $y$ direction). Its Fourier transform, by
\cite[Thm.~1]{B}, is of the form $\widehat{\gamma}=\nu \ast
\delta^{}_{\ZZ}$ with $\nu=\widehat{\gamma}|_{[0,1)}$.  Following the
approach of \cite{Kaku,BGG}, one can easily see that $\nu$ is a purely
singular continuous measure, and it is possible to compute the
corresponding distribution function explicitly.

The absence of a point part of $\nu$ follows by Wiener's lemma. We
provide a proof of the latter in the Appendix that is tailored to our
later needs; compare \cite[Cor.~7.11]{Katz} or
\cite[Sec.~4.6]{Nad}. Define $\Sigma_{1} (N) = \sum_{m=0}^{N-1}
\epsilon(m)^{2}$, and recall that $\lvert \epsilon (n)\rvert \le
\epsilon (0)$ from positive definiteness. Now, using Jensen's
inequality via $(a+b)^2\le 2(a^2+b^2)$, we can estimate
\begin{eqnarray}
     \Sigma_{1} (3N)\, 
      & = & \sum_{r=0}^{2}\sum_{m=0}^{N-1} \epsilon(3m+r)^{2} 
     \nonumber \\
    &  = & \Sigma_{1} (N) + \sum_{m=0}^{N-1} \bigl(-\frac{2}{9}\epsilon(m)
    + \frac{1}{3}\epsilon(m+1)\bigr)^2 + \bigl(\frac{1}{3}\epsilon(m)-
     \frac{2}{9}\epsilon(m+1)\bigr)^2 \\[1mm]
      & \le & \bigl(1+\frac{52}{81}\bigr) \,\Sigma_{1}(N) + 
         \frac{26}{81}\bigl(\epsilon(N)^2-\epsilon(0)^2\bigr)
       \, \le \,\frac{133}{81}\, \Sigma_{1} (N) \, ,
     \nonumber
\end{eqnarray} 
which shows that $\Sigma_{1} (N) =\mathcal{O}(N^{\alpha})$ with
$\alpha=\log_{3}(133/81)<1/2$, so the measure $\nu$ has no points.
The absence of an absolutely continuous component can be shown in
complete analogy to the Thue-Morse case by invoking the
Riemann-Lebesgue lemma, so we conclude that $\nu$ is purely singular
continuous. The corresponding distribution function $\varPhi$ is defined
via $\varPhi (x) = \nu \bigl( [0,x] \bigr)$ for $x\in [0,1]$ and extended
to $\RR$ by means of $\varPhi (x+n) = n + \varPhi (x)$ for $n\in\ZZ$, which
gives a continuous, non-decreasing function on $\RR$. The general
properties of $\varPhi$ follow by the arguments used for the
(generalised) Thue-Morse system; compare \cite{BBGG,BGG} for
details. Let us summarise some of them as follows.

The continuous distribution function $\varPhi$ possesses the Fourier
series representation
\begin{equation}\label{eq:F}
   \varPhi (x) \, = \,x + \sum_{m\ge 1} \frac{\epsilon(m)}{m\pi} 
    \sin(2\pi mx) \, = \,
   x \sum_{m\in\ZZ} \epsilon(m) \, \mathrm{sinc} (2 \pi m x)\, ,
\end{equation}
with $\mathrm{sinc} (z) = \frac{\sin (z)}{z}$ and $\mathrm{sinc}
(0)=1$. The first sum is uniformly (but not absolutely)
convergent. The second equality is a consequence of $\varPhi$ being
the integral of a (formal) Fourier-Stieltjes series with coefficients
$\epsilon(m)$; compare \cite[Sec.~1.2.6]{P} and \cite[Sec.~7]{Katz}.
The limit can be approximated by distribution functions $\varPhi_{n}$
of absolutely continuous positive measures,
\begin{equation}\label{eq:Fn}
  \varPhi^{}_{n}(x) \, = \, x + \sum_{m=1}^{3^{n}-1} 
                \frac{c^{}_{n}(m)}{m\pi}\sin(2\pi mx)\, ,
\end{equation}
where $\varPhi^{}_{0} (x) = x$ (which is the distribution function of
Lebesgue measure $\lambda$) and the coefficients satisfy the initial
conditions $c^{}_{0}(m)=\delta^{}_{m,0}$ together with the recursions
\[
   \begin{split}
   c^{}_{n}(3m)\, & = \, c^{}_{n-1}(m)\\
   c^{}_{n}(3m \! + \! 1)\, & = \, -\frac{2}{9}c^{}_{n-1}(m) + 
                        \frac{1}{3}c^{}_{n-1}(m \! + \! 1)\\
   c^{}_{n}(3m \! + \! 2)\, & = \, \frac{1}{3}c^{}_{n-1}(m) - 
                        \frac{2}{9}c^{}_{n-1}(m \! + \! 1) \, .
   \end{split}
\]
The distribution functions $\varPhi_{n}$ define a sequence of
absolutely continuous measures that converge vaguely to our singular
continuous measure defined by $\varPhi$.  Here, the individual Fourier
series for $\varPhi^{}_{n} (x) - x$ are finite sums, hence
trigonometric polynomials with fundamental period $1$. Moreover, we
have uniform convergence $\varPhi_{n}\longrightarrow\varPhi$ as
$n\to\infty$ by the same argument as in \cite{BGG}, which is based on
the stepping-stone argument from \cite[Thm.~30.13]{Bauer}.

The corresponding densities $\varphi^{}_{n}(x)$, defined by $\dd
\varPhi^{}_{n}(x) =\varphi^{}_{n}(x)\dd x$, have the Riesz product
representation
\begin{equation}\label{eq:fn}
   \varphi^{}_{n}(x) \, = \, \prod_{\ell=0}^{n-1} \theta(3^{\ell}x)\, ,
\end{equation}
where $\theta(x)=1-\frac{4}{9}\cos(2\pi x) + \frac{2}{3}\cos(4\pi x)$
is a strictly positive, $1$-periodic function that satisfies $\theta
(1-x) = \theta (x)$ and $\int_{0}^{1} \theta (x) \dd x = 1$. The
functions $\varPhi^{}_{4}$ and $\varphi^{}_{4}$ are illustrated in
Figure~\ref{fig:sqsubmeas}.  The measure defined by $\varPhi$ is
represented as the infinite Riesz product $\prod_{\ell=0}^{\infty}
\theta (3^{\ell} x)$, to be understood with convergence in the vague
topology. In summary:

\begin{figure}
\begin{center}
  \includegraphics[width=0.45\textwidth]{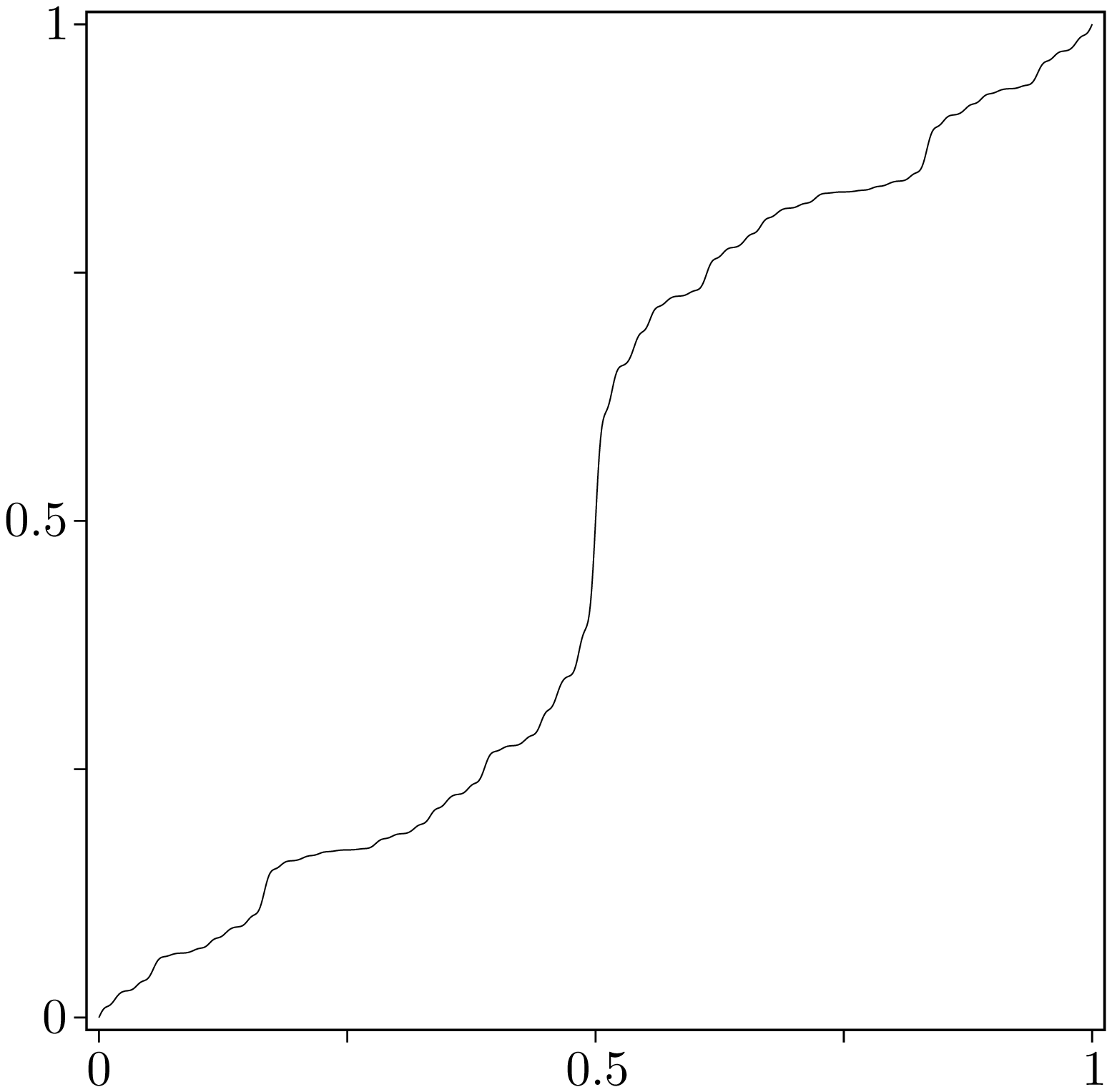}
  \hspace*{0.025\textwidth}
  \includegraphics[width=0.45\textwidth]{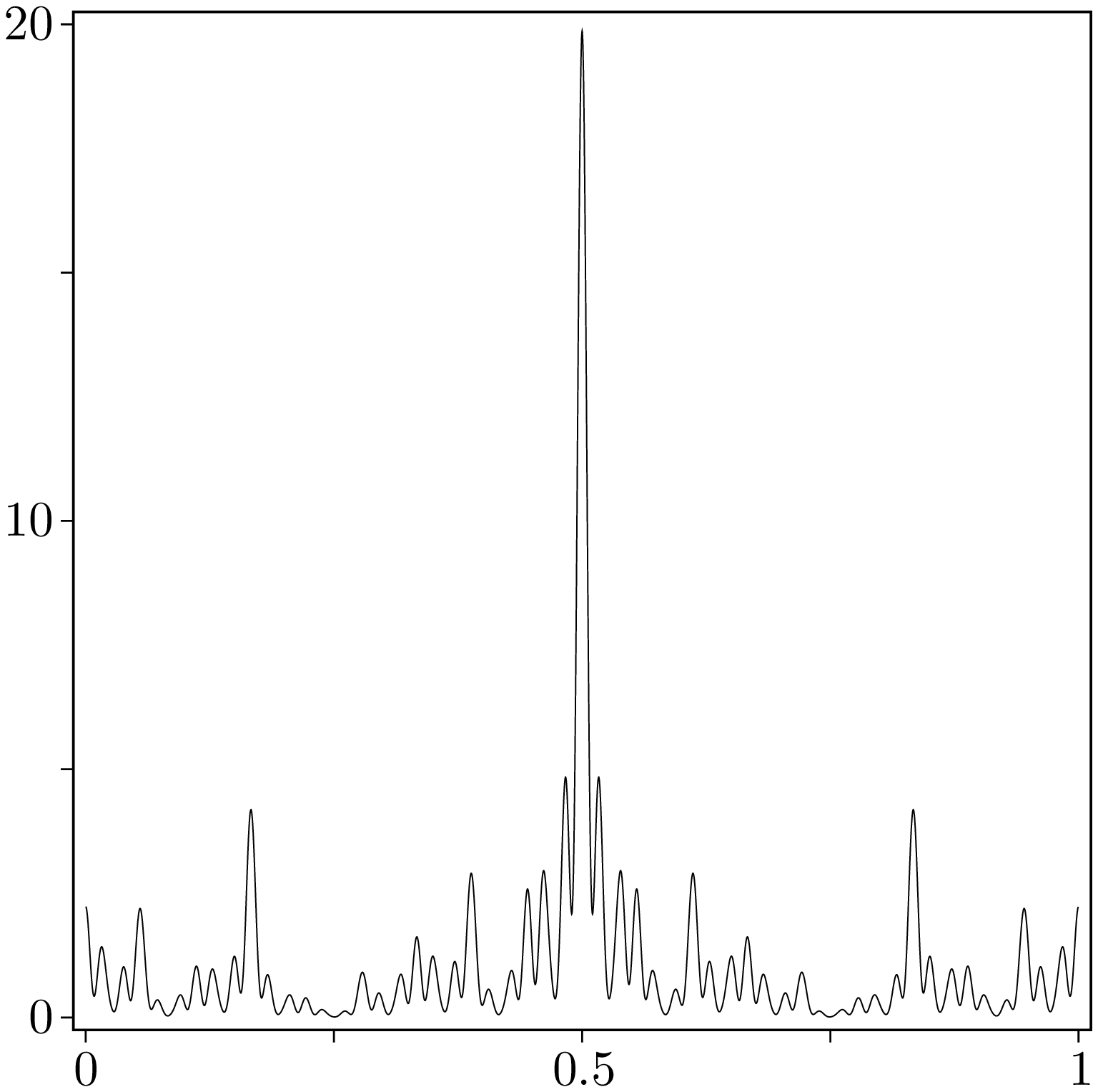}
\end{center}
\caption{\label{fig:sqsubmeas}Distribution function $\varPhi_{4}(x)$
  (left) and its smooth density $\varphi_{4}(x)$ (right), according to
  Eqs.~\eqref{eq:Fn} and \eqref{eq:fn}.}
\end{figure}

\begin{prop}\label{prop:1D}
  The function\/ $\epsilon \! : \, \ZZ \longrightarrow \RR$ defined by
  Eq.~\eqref{eq:eps-rec} together with\/ $\epsilon(0) =1$ is positive 
  definite. It is the Fourier transform of a probability measure\/ $\nu$
  on\/ $[0,1)$ that is purely singular continuous, with distribution
  function\/ $\varPhi$ of Eq.~\eqref{eq:F} and a representation as an
  infinite Riesz product.   \qed
\end{prop}

\section{Analysis of the planar case}

Since all ingredients to the above arguments are also available in
higher dimensions, we proceed constructively. Here, with the
coefficients from Lemma~\ref{lem:rec}, we define
\begin{equation}\label{eq:sig-def}
   \Sigma (N) \, := \, \sum_{m,n=0}^{N-1}   \langle m,n\rangle^2 ,
\end{equation}
which is now a sum over $N^2$ terms.

\begin{lemma}\label{lem:W-crit}
  The sums of Eq.~\eqref{eq:sig-def} satisfy\/ $\Sigma (3N) \le
   \frac{319}{81}\, \Sigma (N)< 4 \Sigma (N)$.
\end{lemma}

\begin{proof}
The recursion relations \eqref{eq:eta-rec} imply that
\[
  \begin{split}
    \Sigma (3N) \; & = \, \sum_{m,n=0}^{3N-1} \langle m,n \rangle^{2}
       \; =  \sum_{m,n=0}^{N-1}\, \sum_{r,s=0}^{2}
       \langle 3m\! + \! r, 3n \! + \! s\rangle^{2} \\
    & = \, \frac{1}{81} \sum_{m,n=0}^{N-1} 84 \langle m,n \rangle^{2}
     + 14 \langle m,n \! + \! 1\rangle^{2} + 
       14 \langle m\! + \! 1,n \rangle^{2}
     + 7 \langle m\! + \! 1, n \! + \! 1 \rangle^{2} \\
    & \hphantom{ = \, \frac{1}{81} \sum_{m,n=0}^{N-1}}\;
     + 2 \bigl( 4 \langle m,n \rangle - 2 \langle m,n\! +\! 1 \rangle
     - 2 \langle m\! + \! 1,n \rangle  + \langle m\! + \! 1,
       n \! + \! 1 \rangle \bigr)^{2}.
  \end{split}
\]
One can now use Jensen's inequality in the form $(a+b+c+d)^{2} \le
4\ts (a^2 + b^2 + c^2 + d^2)$ and observe the relation
\[
   \sum_{m,n=0}^{N-1} \langle m,n \! + \! 1\rangle^{2} 
   \, = \,
   \Sigma (N) + \sum_{m=0}^{N-1} \bigl( \langle m,N \rangle^{2}
    - \langle m,0 \rangle^{2} \bigr) \, \le \, \Sigma (N) \ts ,
\]
which follows from Lemma~\ref{lem:pos}. The analogous estimate
holds for $\sum_{m,n=0}^{N-1} \langle m \! + \! 1,n\rangle^{2}$,
while
\[
  \begin{split}
   \sum_{m,n=0}^{N-1}  \langle m \! + \! 1, & n \! + \! 1\rangle^{2}
   \, = \, \Sigma (N) + \bigl(\langle N,N \rangle^{2} 
           - \langle 0,0 \rangle^{2}\bigr) \\
   & + \sum_{m=1}^{N-1} \bigl(\langle m,N \rangle^{2} 
           - \langle m,0 \rangle^{2}\bigr)
     + \sum_{n=1}^{N-1} \bigl(\langle N,n \rangle^{2} 
           - \langle 0,n \rangle^{2}\bigr) \, \le \, \Sigma (N) \ts .
  \end{split}
\]
Putting everything together leads to the estimate
\[
   \Sigma (3N) \, \le \, \frac{319}{81}\, \Sigma (N) 
\]
which implies our claim. 
\end{proof}

\begin{remark}
Let us note that, without Lemma~\ref{lem:pos}, one would obtain 
\begin{equation}\label{eq:alt-estimate}
   \Sigma (3 N) \, \le \, \frac{319}{81}\, \Sigma (N) 
            +  \mathcal{O} (N) \ts ,
\end{equation}
which is weaker but still sufficient for the application of Wiener's
lemma later on. This is the type of relation that one can expect in
similar examples that fail to satisfy Lemma~\ref{lem:pos}.
\end{remark}

\begin{lemma}\label{lem:RL}
  Let\/ $\zeta \! : \, \ZZ^{2} \longrightarrow \RR$ be a function that
  satisfies the recursion relations \eqref{eq:eta-rec}, with\/ $\zeta
  (0,0) \ge 0$. Then, $\zeta$ is a positive definite function on\/
  $\ZZ^{2}$, and defines a unique positive measure\/ $\mu^{}_{\zeta}$
  on the\/ $2$-torus\/ $\TT^{2}$, via
\[
     \zeta (m,n) \, =  \int_{\TT^{2}} \ee^{2 \pi \ii k x} 
     \dd \mu^{}_{\zeta} (x) \ts ,
\]
  where\/ $\mu^{}_{\zeta} (\TT^{2}) = \zeta (0,0)$.
  Moreover, the measure\/ $\mu^{}_{\zeta}$ is absolutely continuous,
  relative to the Haar measure on\/ $\TT^{2}$, if and only if\/
  $\zeta (0,0) = 0$. In the latter case, $\zeta \equiv 0$.
\end{lemma}

\begin{proof}
  When $\zeta (0,0) = 1$, we have $\zeta(m,n) = \langle m,n \rangle$
  by Lemma~\ref{lem:rec}, which is positive definite by construction.
  Since the recursion \eqref{eq:eta-rec} is linear, $\zeta (0,0) = a
  \ge 0$ leads to $\zeta (m,n) = a\, \langle m,n \rangle$, which is
  still positive definite on $\ZZ^{2}$. The Herglotz-Bochner theorem
  \cite{Rudin} results in the representation via the unique positive
  measure $\mu^{}_{\zeta}$; compare Lemma~\ref{lem:gen-Wiener} below.
  It is a probability measure if and only if $\zeta (0,0) = 1$.

Whenever $a>0$, the special values $\zeta(m,n)$ with $-2 \le m,n \le
2$ are different from $0$, again by Lemma~\ref{lem:rec}, wherefore
$\zeta (3m,3n) = \zeta (m,n)$ transports them all the way to infinity.
By the Riemann-Lebesgue lemma \cite{P,Rudin}, this implies that
$\mu^{}_{\zeta}$ cannot be absolutely continuous relative to Lebesgue
measure (which is the Haar measure on $\TT^{2}$). The only exception
is $\zeta (0,0)=0$, which forces also all special values to vanish,
and hence $\zeta$ itself.
\end{proof}

\begin{theorem}
  The diffraction measure\/ $\widehat{\gamma} =
  \widehat{\gamma^{}_{\omega}}$ of the balanced Dirac comb\/ $\omega$
  of Eq.~\eqref{eq:def-comb} is a translation bounded, positive
  measure that is purely singular continuous. All elements of the hull
  of\/ $\omega$ possess the same autocorrelation and the same
  diffraction measure.
\end{theorem}
\begin{proof}
  The diffraction measure is $\widehat{\gamma} = \mu *
  \delta^{}_{\ZZ^2}$, by an application of \cite[Thm.~1]{B},
  where $\mu$ is the positive measure from
  Eq.~\eqref{eq:mu-def}.  From Lemma~\ref{lem:W-crit}, we know that $
  \frac{1}{N^2}\Sigma (N) \longrightarrow 0$ as $N\to\infty$, so that
  Wiener's lemma (see Lemma~\ref{lem:gen-Wiener} in the appendix)
  implies that $\mu$ is continuous.

  We may now employ the unique decomposition $\mu =
  \mu^{}_{\mathsf{sc}} + \mu^{}_{\mathsf{ac}}$ as a sum of
  non-negative measures, relative to Lebesgue
  measure on $\TT^{2}$.  Defining the functions
  $\eta^{}_{\mathsf{ac}}$ and $\eta^{}_{\mathsf{sc}}$ as the inverse
  Fourier transforms of $\mu^{}_{\mathsf{ac}}$ and
  $\mu^{}_{\mathsf{sc}}$, in analogy to Eq.~\eqref{eq:mu-def},
  one finds 
\[
    \eta^{}_{\mathsf{ac}} (m,n) +
    \eta^{}_{\mathsf{sc}} (m,n) \, = \, 
    \langle m,n \rangle
\]
for all $(m,n) \in \ZZ^{2}$, together with $\eta^{}_{\mathsf{ac}}
(0,0) \ge 0$ and $\eta^{}_{\mathsf{sc}} (0,0) \ge 0$. Since
$\mu^{}_{\mathsf{ac}}$ and $\mu^{}_{\mathsf{sc}}$ are mutually
orthogonal in the measure sense, it is clear that both
$\eta^{}_{\mathsf{ac}}$ and $\eta^{}_{\mathsf{sc}}$ satisfy the same
set of recursions, namely those of Eq.~\eqref{eq:eta-rec}, but
possibly with different initial conditions.  By an application of
Lemma~\ref{lem:RL}, we see that $\eta^{}_{\mathsf{ac}}$ must vanish,
so that $\langle m,n \rangle = \eta^{}_{\mathsf{sc}} (m,n)$ on
$\ZZ^{2}$. This implies $\mu$, and hence $\widehat{\gamma}$, to be
purely singular continuous.

The hull of $\omega$, which is its orbit closure in the vague
topology, is strictly ergodic. Consequently, the autocorrelation
coefficients of each Dirac comb in this hull are the ones specified by
Eq.~\eqref{eq:eta-rec} and Lemma~\ref{lem:rec}, which implies the last
claim.
\end{proof}

To further investigate the probability measure $\mu$, and then also
the diffraction measure $\widehat{\gamma} = \mu *
\delta^{}_{\ZZ^{2}}$, let us define a multi-dimensional distribution
function $F$, compare \cite[\S 23]{Gne}, as $F(x,y) := \mu
\bigl( [0,x) \times [0,y) \bigr)$, which is continuous from the left
and from below by construction; see \cite[Sec.~I.6]{Bauer} or
\cite[Sec.~12]{Bill} for an alternative, equivalent approach. Note that
$F$ is non-decreasing along any line parallel to the $x$-axis or the
$y$-axis. The (continuous) measure $\nu$ from
Proposition~\ref{prop:1D} is the marginal of $\mu$ in the sense that
\begin{equation}\label{eq:marginal}
   \nu \bigl( [0,x) \bigr) \, = \,
   \mu \bigl( [0,x) \times [0,1) \bigr)
   \quad \text{and} \quad 
   \mu \bigl( [0,1) \times [0,y) \bigr)
   \, = \,  \nu \bigl( [0,y) \bigr)
\end{equation}
holds for all $x,y \in [0,1)$, where the second relation follows
by the symmetry of our system. Since the distribution function of
$\nu$ is continuous by Proposition~\ref{prop:1D}, an application
of \cite[Thm.~2.3]{Tucker} shows that $F$ is a continuous function.
This is one of the somewhat subtle points to observe when using
higher-dimensional distribution functions.

To continue, it is advantageous to extend $F$ to a continuous function
on $\RR^{2}$, which is most easily done via $F(x,y) = \widehat{\gamma}
\bigl( [0,x) \times [0,y) \bigr)$ for $x,y \ge 0$. Here, $F$ is
clearly continuous, by an extension of the above argument across 
the lines $\{ x = m\}$ and $\{ y = n\}$, where $F(m,y)=\varPhi(y)$
and $F(x,n) = \varPhi(x)$. In particular, we thus have
\[
    F (x,y) \, = \, \widehat{\gamma} 
    \bigl( [0,x] \times [0,y] \bigr)
\]
for $x,y \ge 0$. This can now consistently be extended to all of
$\RR^{2}$ by setting $F(-x,y) = F(x,-y) = - F(x,y)$ and hence
$F(-x,-y) = F(x,y)$.  In particular, one has $F(0,0)=0$ as well as
$F(0,y) = F(x,0)=0$, and $F$ is continuous on $\RR^{2}$.

\begin{lemma}\label{lem:F-prop}
  Let\/ $F$ be the continuous distribution function defined by\/
  $F(x,y) = \widehat{\gamma} \bigl([0,x]\times [0,y]\bigr)$
  for\/ $x,y \ge 0$  together with\/ $F(-x,y) =
  F(x,-y) = - F(x,y)$. Then, $F$ satisfies
\[
   F(x+1,y) \, = \, F(x,y) + \varPhi (y)
   \quad \text{and} \quad
   F(x,y+1) \, = \, F(x,y) + \varPhi (x)
\]
for arbitrary\/ $x,y \in \RR$, where\/ $\varPhi$ is the function from
Eq.~\eqref{eq:F} and Proposition~$\ref{prop:1D}$.  
The function\/ $h$ defined by\/ $ h(x,y) = F(x,y) - x\ts \varPhi(y) -
y\ts \varPhi(x) + xy$ is continuous and\/ $\ZZ^{2}$-periodic.
\end{lemma}

\begin{proof}
  Observe first that Eq.~\eqref{eq:marginal} has a natural extension
  the periodic measures $\nu * \delta^{}_{\ZZ}$ and
  $\widehat{\gamma} = \mu * \delta^{}_{\ZZ^{2}}$.

  Now, for $x,y \ge 0$, one has
\[
    F(x,y+1) - F(x,y)\, = 
    \int_{0}^{x}\! \int_{y}^{y+1}\! \dd \widehat{\gamma} \, =  
    \int_{0}^{x}\! \int_{ \{ y \} }^{ \{ y \}+1}\! \dd \widehat{\gamma} \, =  
    \int_{0}^{x}\! \int_{0}^{1}\! \dd \widehat{\gamma} \, =  
    \int_{0}^{x}\! \dd (\nu * \delta^{}_{\ZZ}) \, = \, \varPhi (x)\ts ,
\] 
  where $\{ x\}$ denotes the fractional part of $x$, and the
  $1$-periodicity of $\widehat{\gamma}$ (in $y$) was used twice,
  in the second step via $[ \{ y \}, \{ y \} +1 ] = 
  [ \{ y \}, 1 ] \cup ( 1, \{ y \} + 1]$ together with the
  fact that $\widehat{\gamma}$ is a continuous measure.
  The other identity follows analogously, and the extension to
  all of $\RR^{2}$ is clear by symmetry.

  The final claim now follows from a simple calculation.
\end{proof}

Due to the underlying symmetry, $h$ has a Fourier series of the form
\[
   h (x,y) \, =  \sum_{m,n=1}^{\infty} a^{}_{m,n}
   \sin (2 \pi m x) \sin (2 \pi n y) \ts ,
\]
where $a^{}_{m,n} = \frac{\langle m,n \rangle}{\pi^{2} mn}$ follows
from a routine calculation; see \cite{AAP} for
background on multiple Fourier series. 
Together with Eq.~\eqref{eq:F}, this leads
to the series representation
\begin{equation}\label{eq:F-series}
\begin{split}
   F(x,y) \, =&\, - xy + x \ts \varPhi(y) \ts + y\ts \varPhi(x)
   \, + \! \sum_{m,n=1}^{\infty} 
   \frac{\langle m,n\rangle }{\pi^2 \ts mn}
   \sin (2 \pi m x) \sin (2 \pi n y) \\
   = & \;\, xy + x \sum_{n=1}^{\infty} 
        \frac{\langle 0,n\rangle }{\pi\ts n}
   \sin (2 \pi n y) + y \sum_{m=1}^{\infty} 
        \frac{\langle m,0\rangle }{\pi\ts m}
   \sin (2 \pi m x)\\ &\; + \sum_{m,n=1}^{\infty} 
   \frac{\langle m,n\rangle }{\pi^2 \ts mn}
   \sin (2 \pi m x) \sin (2 \pi n y)\ts ,
\end{split}
\end{equation}
which is actually uniformly convergent (for summation over
square-shaped regions). In line with Eq.~\eqref{eq:F}, one can also
rewrite the distribution function as
\[
    F(x,y) \, = \, xy \sum_{m,n\in\ZZ} \langle m,n \rangle\,
    \mathrm{sinc} (2 \pi m x)\, \mathrm{sinc} (2 \pi n y) \ts ,
\]
which highlights the structure as an integral over a planar
Fourier-Stieltjes series.

It is important to observe that $F$ is well-defined once the measure
$\mu$ is given, and that $F$ specifies the measure uniquely;
compare \cite[Thm.~12.5]{Bill} for details. We can thus employ
the Lebesgue-Stieltjes approach to measures also in this more
general situation.

For actual calculations, it is advantageous to use an approximation to
$F$ via a sequence of distribution functions with densities, pretty
much as in the Thue-Morse example. Indeed, the recursion relations
\eqref{eq:eta-rec}, via an explicit but somewhat tedious calculation,
implies the functional relation
\begin{equation}\label{eq:F-rel}
   F (x,y)  \, = \, \frac{1}{9} \int_{0}^{3x}
      \int_{0}^{3y} \vartheta \bigl( \frac{x}{3}, \frac{y}{3}\bigr)
      \dd F (x,y)
\end{equation}
(written in Lebesgue-Stieltjes notation) with the kernel function
\begin{equation}\label{eq:vartheta-def}
   \vartheta (x,y) \, = \,  \frac{1}{9} \bigl(1
   + 2 \cos (2\pi\ts x) + 2 \cos (2\pi\ts y)
   - 2 \cos (2\pi (x \! + \! y)) 
   - 2 \cos (2\pi (x \! - \! y))\bigr)^{2}\ts ,
\end{equation}
which is shown in Figure~\ref{fig:theta}.  Clearly, $\vartheta$ is a
non-negative, $\ZZ^{2}$-periodic function that is symmetric in both
arguments. Moreover, it satisfies the relations $\vartheta (1 \! - \!
x,y) = \vartheta (x,y) = \vartheta (x,1 \! - \! y)=\vartheta(y,x)$
together with the normalisation $\int_{[0,1]^{2}} \vartheta(x,y) \dd x
\dd y = 1$. A general formula for $\vartheta$ will be discussed below
in Eq.~\eqref{eq:gen-theta}.
 
\begin{figure}
\begin{center}
  \includegraphics[width=0.6\textwidth]{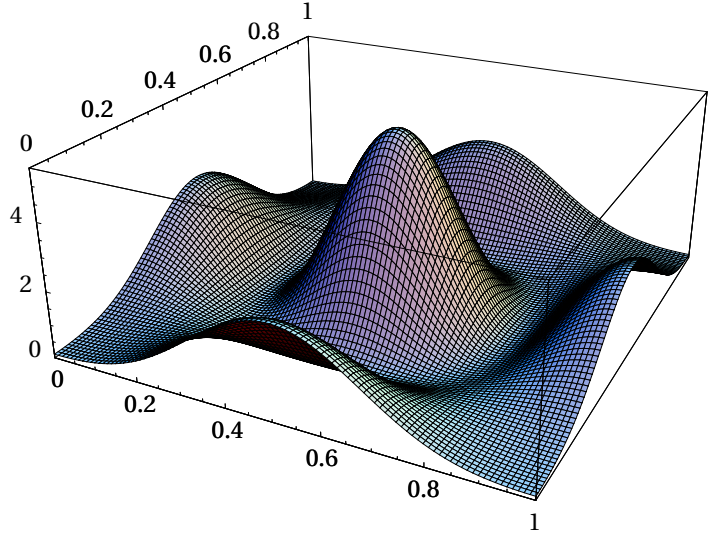}
\end{center}
\caption{\label{fig:theta} The function $\vartheta$ of
  Eq.~\eqref{eq:vartheta-def} on $[0,1]^{2}$.}
\end{figure}

The functional relation \eqref{eq:F-rel} can now be employed to define
an iterative calculation of $F$ as follows. One starts from $F^{(0)}
(x,y) = xy$ (which corresponds to the measure $\dd F^{(0)} = \lambda$)
and continues with the iteration
\begin{equation}\label{eq:F-iter}
   F^{(N+1)} (x,y) \, = \, \frac{1}{9} \int_{0}^{3x}
      \int_{0}^{3y} \vartheta \bigl( \frac{x}{3}, \frac{y}{3}\bigr)
      \dd F^{(N)} (x,y)\ts .
\end{equation}
The functions $F^{(N)}$ have the form
\[
\begin{split}
    F^{(N)} (x,y) \, =& \;\, x\ts y + x \sum_{n=1}^{3^{N}-1} 
   \frac{\beta^{(N)}_{0,n}}{\pi\ts n} \sin (2 \pi n y) + 
   y \sum_{m=1}^{3^{N}-1} \frac{\beta^{(N)}_{m,0}}{\pi\ts m}
   \sin (2 \pi m x) \\ & \; + \sum_{m,n=1}^{3^{N}-1} 
   \frac{\beta^{(N)}_{m,n}}{\pi^2\ts mn}
   \sin (2 \pi m x) \sin (2 \pi n y)\ts ,
\end{split}
\]
where the coefficients $\beta^{(N)}_{m,n}$ are defined for $N,m,n\ge
0$ by the initial conditions $\beta^{(0)}_{m,n}=\delta^{}_{m,0}
\delta^{}_{n,0}$ together with the recursion 
\[
    \beta^{(N+1)}_{3m + r, 3n +s} \, =  
   \sum_{k = 0}^{\min(1,r)} \sum_{\ell=0}^{\min(1,s)}
   \alpha^{(r,s)}_{k,\ell}\ts \beta^{(N)}_{m+k,n+\ell} \ts ,
\]
with $N\ge 0$ and the same coefficients $\alpha^{(r,s)}_{k,\ell}$ as
in Eq.~\eqref{eq:eta-rec-2}. This can once again be verified by a
direct calculation.

All $F^{(N)}$ represent absolutely continuous measures, so that we can
define Radon-Nikodym densities via $\dd F^{(N)} (x,y) =
f^{(N)} (x,y) \dd x \dd y$, where 
\[
   f^{(N)} (x,y) \, = \, \frac{\partial^{2}}
   {\partial x \ts \partial y} F^{(N)}(x,y) \ts .
\]
This gives $f^{(1)} = \vartheta$, with the function $ \vartheta $ from
Eq.~\eqref{eq:vartheta-def}, via the application of some trigonometric
identities. The iteration now reveals the Riesz product formula
\begin{equation}\label{eq:riesz-2}
   f^{(N)} (x,y) \, = \prod_{\ell=0}^{N-1} \vartheta (3^{\ell}x,3^{\ell}y) \ts ,
\end{equation}
which highlights the deeper role of $\vartheta$ in
Eq.~\eqref{eq:F-iter}. The distribution function $F^{(3)}$ and the
corresponding density $f^{(3)}$ are shown in
Figure~\ref{fig:distdens}.

\begin{figure}
\begin{center}
  \includegraphics[width=0.45\textwidth]{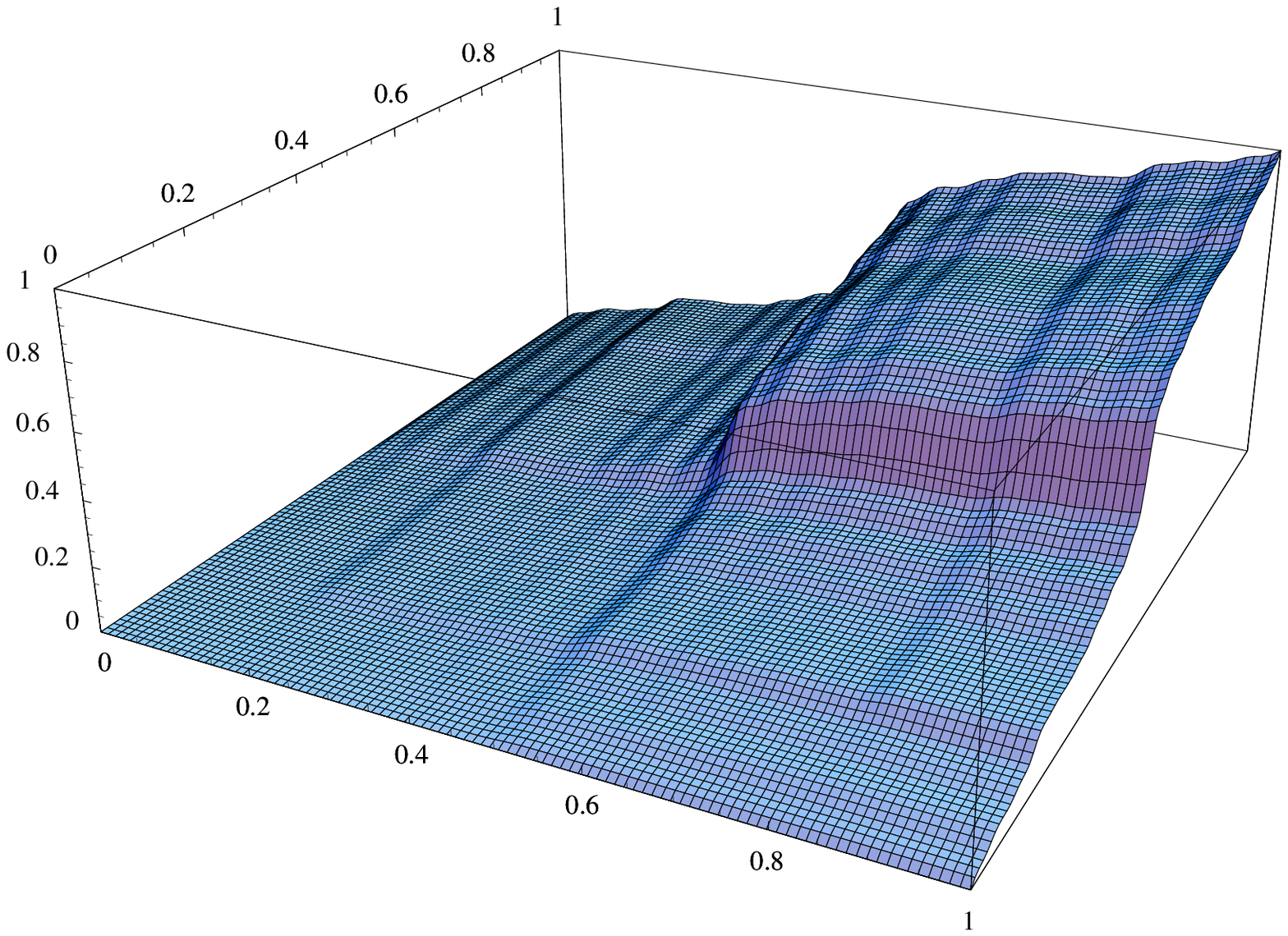}
  \includegraphics[width=0.45\textwidth]{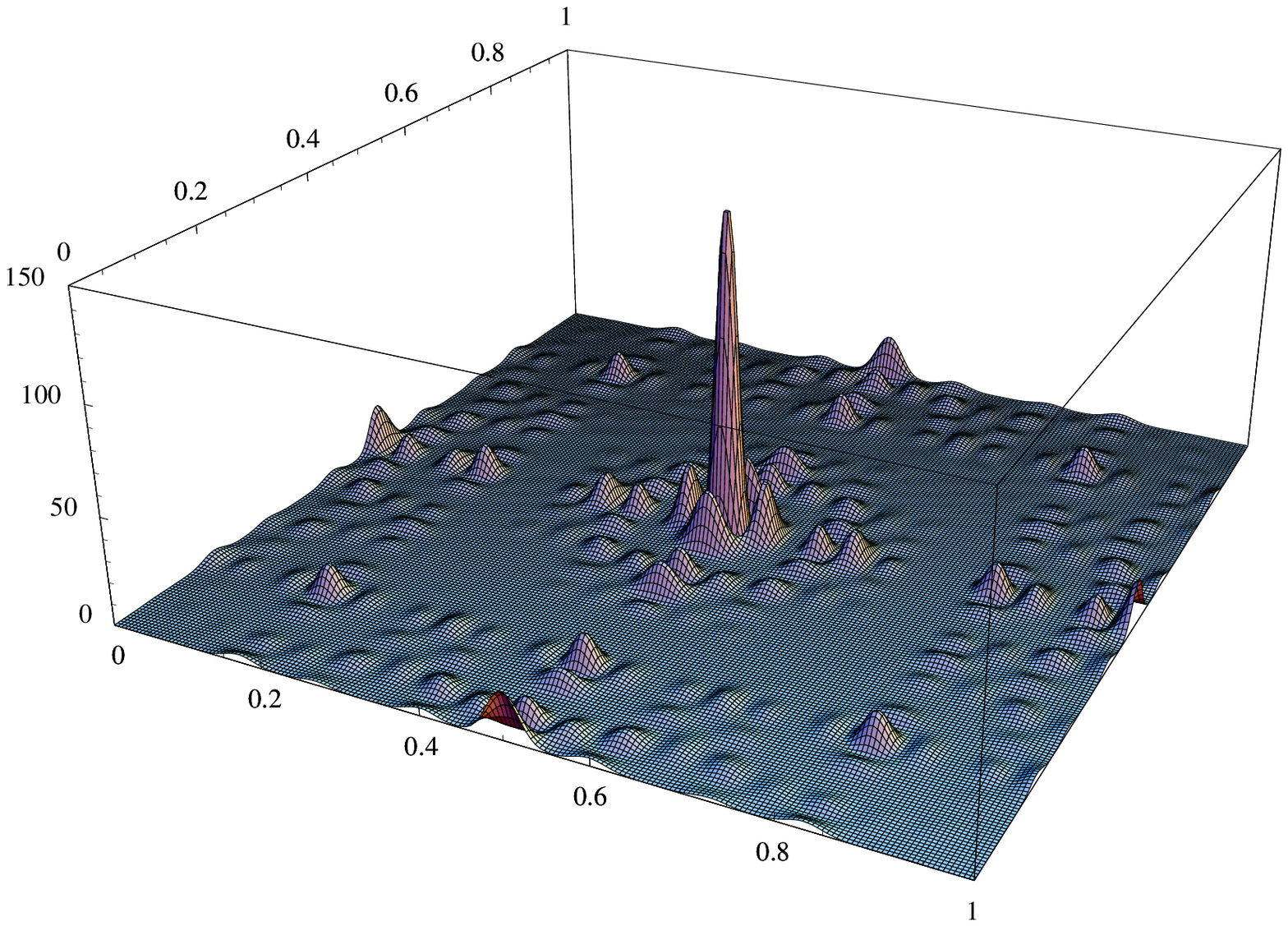}
\end{center}
\caption{\label{fig:distdens} The distribution function $F^{(3)}$ 
  of Eq.~\eqref{eq:F-iter} (left) and the corresponding density $f^{(3)}$ 
  of Eq.~\eqref{eq:riesz-2} (right), approximating
  the diffraction measure of the squiral tiling (in the version
  of Figure~\ref{fig:blocktil})  on $[0,1]^{2}$.}
\end{figure}

\begin{remark}
  The measure defined by the continuous function $F$ is represented by
  the infinite Riesz product $\prod_{\ell=0}^{\infty} \vartheta
  (3^{\ell} x, 3^{\ell} y)$, to be understood with convergence in the
  vague topology. By a comparison of Eqs.~\eqref{eq:F-rel} and
  \eqref{eq:F-iter}, it is clear that $F$ is a fixed point of the
  latter. In fact, within the class of distribution functions with
  certain continuity and additivity constraints, it is the only fixed
  point, with uniform (but not absolute) convergence towards it from
  the initial condition $F^{(0)}$, again by arguments analogous to
  \cite[Thm.~30.13]{Bauer}.
\end{remark}

\section{A topological factor with maximal pure point spectrum}

It is well known that the one-dimensional Thue-Morse system admits a
topological factor with maximal pure point spectrum, which can be
defined by the period doubling substitution. In fact, the latter is
induced by a simple sliding block map of width $2$. The corresponding
factor map is globally $2$-to-$1$ between the hulls; see \cite{BGG}
for details and an extension to generalised Thue-Morse sequences. 
Amazingly, a
similar approach also works for the squiral tiling.

Define the mapping $\psi\! :\, \{1,\bar{1}\}^{\ZZ^{2}} \longrightarrow
\{1,\bar{1}\}^{\ZZ^{2}}$ by $w\mapsto\psi w$ with 
\[
   (\psi w)_{m,n} \, = \, w_{m,n}w_{m+1,n}w_{m,n+1}w_{m+1,n+1}\ts , 
\]
which is continuous. It is clear that $\YY:=\psi\XX$ defines a factor
for the action of $\ZZ^{2}$. Moreover, the inflation rule
\eqref{eq:sqinfl} induces a new inflation on $\YY$, which reads
\begin{equation}\label{eq:kroninfl}
   \begin{matrix}
     & & 1 & 1 & a \\
     & & \bar{1} & \bar{1} & 1 \\
     a & \longmapsto & \bar{1} & \bar{1} & 1 
     \end{matrix}
\end{equation}
with $a\in\{1,\bar{1}\}$. Note that the arrangement matches that of
the inflation rule \eqref{eq:sqinfl}. One checks consistency by
starting from the $14$ legal patches with $2\times 2$ squares, and
verifies that they inflate to larger patches that produce the correct
$3\times 3$ blocks in the lower left corner under $\psi$. One can
check that the block map is \emph{not} globally $2$-to-$1$, though it
is $2$-to-$1$ almost everywhere. For a more detailed discussion and a
classification of the substitution factors of the squiral, we refer to
\cite{BGG2}.

Consider one of the two possible fixed points of the induced
substitution \eqref{eq:kroninfl} and decompose
$\ZZ^{2}=\vL_{+}\dotcup\vL_{-}$, where
$\vL_{\pm}=\{(m,n)\in\ZZ^{2}\mid w_{m,n}=\pm 1\}$. Using
$e^{}_{1}=(1,0)$ and $e^{}_{2}=(0,1)$, the fixed point property now
induces the set-valued relations
\begin{equation}
   \vL_{\pm} \, = \, \bigl(3\ZZ^{2}+S_{\pm}\bigr) \dotcup \,
   \bigl(3\vL_{\pm}+2(e^{}_{1}\!+\!e^{}_{2})\bigr)
\end{equation}
with the translation sets $S_{+}=\{2e^{}_{1}, 2e^{}_{2},
2e^{}_{1}\!+\!e^{}_{2}, e^{}_{1}\!+\!2e^{}_{2}\}$ and $S_{-}=\{0, e^{}_{1},
e^{}_{2}, e^{}_{1}\!+\!e^{}_{2}\}$. 

The solutions are
\[
   \vL_{\pm} \, = \, 
    \bigdotcup_{n\ge 1} \bigl( 3^n \ZZ^{2} + 3^{n-1} S_{\pm}
    + (3^{n-1}-1) (e^{}_{1}\!+\!e^{}_{2})\bigr)
    \;\dotcup\, A_{\pm}
\]
where either $A_{+}=\{-(e^{}_{1}\!+\!e^{}_{2})\}$ and
$A_{-}=\varnothing$, or vice versa. The special role of the point
$(-1,-1)$ becomes transparent in the topology of the $3$-adic numbers.

The sets $\vL_{\pm}$ are model sets with two copies of the $3$-adic
numbers as internal space. The Dirac combs $\delta^{}_{\!\vL_{\pm}}$
are thus pure point diffractive by the model set theorem; see
\cite[Thm.~4.5]{M} or \cite[Thm.~2]{BM}. A standard calculation
now reveals that the corresponding Fourier module is
$\ZZ[\frac{1}{3}]\times \ZZ[\frac{1}{3}]$. By general arguments
\cite{LMS,BL}, this equals the dynamical spectrum of $\YY$.

By the results of \cite{Nat1}, the pure point part of the dynamical
spectrum of $\XX$ thus equals the dynamical spectrum of $\YY$, which
shows that $\YY$ is indeed a topological factor that exhausts the pure
point spectrum of the squiral tiling. To summarise, we have shown
the following.

\begin{theorem}
  The squiral tiling dynamical system has a dynamical spectrum of
  mixed type, with pure point part\/ $\ZZ[\frac{1}{3}] \times
  \ZZ[\frac{1}{3}]$ and an additional singular continuous part.  The
  latter is exploited by the diffraction measure of the balanced
  Dirac comb of Eq.~\eqref{eq:def-comb}, while the former can be
  recovered from the pure point diffraction spectrum of the factor
  system defined via Eq.~\eqref{eq:kroninfl}.  \qed
\end{theorem}

\section{Generalisations}\label{sec:gen}

When inspecting the structure of the above constructive approach, it
is apparent that the method is not restricted to the particular
example of the squiral tiling or its equivalent block substitution
rule. Instead, one might as well consider a $d$-dimensional
extension to bijective block substitutions on a binary
alphabet as follows. 

Let $K$ be a matrix array of dimension $K_1 \times K_2 \times \ldots
\times K_d$, with all $K_i \ge 2$ and entries $\pm 1$ (with $\bar{1} =
-1$ as before), not all equal.  Now consider the rule
\[
   \varrho \! : \;  1 \longmapsto K \ts , \;
    \bar{1} \longmapsto \bar{K} \ts ,
\]
where $\bar{K}$ is obtained from $K$ by flipping all signs.  This
defines a primitive block substitution rule with substitution
matrix $M=\left(\begin{smallmatrix} m & \bar{m} \\
    \bar{m} & m \end{smallmatrix}\right)$, where $m$ is the
cardinality of $1$s in $K$ and $\bar{m}$ that of $\bar{1}$s. The PF
eigenvalue is $\prod_{i} K_i$, with (left and right) eigenvector given
by $(1,1)$. This is in line with a tiling interpretation, with unit
cubes of two colours as prototiles. They occur with equal
frequency in any fixed point tiling.

Select any legal block patch of size $2\!\times\! 2 \!\times\!  \ldots
\!\times\!  2$, which contains $2^{d}$ blocks, and take its centre as
reference point (meaning that each sector of $\ZZ^{d}$ contains one
block of it). This way, the origin is a corner of a block.  Starting
from such a seed, we can now iterate $\varrho$ to create a sequence of
block patches that ultimately cover all positions of $\ZZ^{d}$, due to
our assumption that all $K_{i} \ge 2$. Note that the block patches of
type $K$ and $\bar{K}$ fit together in a face-to-face manner in this
process.

Without loss of generality, we may now assume that we have constructed
a fixed point in this way, which is possibly only true after replacing
$\varrho$ by a suitable power of it, and modifying $K$
accordingly. This modification is always possible, as follows from an
application of Dirichlet's pigeon hole principle to the finite set
of legal seeds. Note that any power of $\varrho$ still satisfies our
basic assumptions, wherefore this step is immaterial for the arguments
to follow (although it might technically complicate explicit
calculations due to the increased size of the new array $K$).

Let $w$ be the fixed point under consideration. It is then specified
by the recursion relations
\begin{equation}\label{eq:gen-rec}
    w^{}_{K^{}_{1} m^{}_{1} + r^{}_{1}, \ldots, K^{}_{d} m^{}_{d} + r^{}_{d}}
    \, = \,\ts \kappa^{}_{r^{}_{1},\ldots, r^{}_{d}} 
    \ts w^{}_{m^{}_{1},\ldots, m^{}_{d}}
\end{equation}
where all $m_i \in \ZZ$ and $0\le r_i < K_i$, while the coefficients
$\kappa^{}_{r^{}_{1},\ldots, r^{}_{d}} \in \{\pm 1\}$ are the elements
of $K$, numbered accordingly. Due to the primitivity of $\varrho$,
the fixed point $w$ defines a repetitive configuration in the
shift space $\{\pm 1\}^{\ZZ}$, wherefore it leads to a hull
\[
    \XX \, = \, \XX (w) \, := \,
    \overline{\{ t+w \mid t\in\ZZ^{d} \}}
\]
that is a minimal subshift of $\{\pm 1\}^{\ZZ}$. Moreover,
$(\XX,\ZZ^{d})$ defines a strictly ergodic dynamical system
by standard arguments \cite{Rob,Sol,Nat1,Q}.

If $\omega = w \ts \delta^{}_{\ZZ^{d}}$ is the attached Dirac comb,
its autocorrelation $\gamma = \omega \circledast \widetilde{\omega}$
(defined in complete analogy to Eq.~\eqref{eq:gam-def}) exists and is
of the form $\gamma = \eta\ts \delta^{}_{\ZZ^{d}}$ with coefficients
\begin{equation}\label{eq:etagen}
   \eta\ts (m^{}_{1},\ldots,m^{}_{d}) \, = 
   \lim_{N\to\infty} \frac{1}{(2N+1)^{d}}
   \sum_{k^{}_{1}=-N}^{N} \dots \sum_{k^{}_{d}=-N}^{N}
   w^{}_{k^{}_{1},\ldots,k^{}_{d}} \, 
   w^{}_{k^{}_{1}-m^{}_{1},\ldots,k^{}_{d}-m^{}_{d}} \ts .
\end{equation}
The sums are orbit averages of a continuous function under the
action of $\ZZ^{d}$, which always exist due to unique ergodicity
(by an application of the stronger version of Birkhoff's ergodic
theorem \cite{W}).

One can now repeat the constructive calculation of the squiral
block substitution in this more general setting. 

\begin{lemma}\label{lem:gen-eta}
  Let\/ $\varrho$ be a primitive, bijective block substitution on a
  binary alphabet, with a fixed point configuration $w$ that satisfies
  Eq.~\eqref{eq:gen-rec}. The corresponding autocorrelation
  coefficients satisfy\/ $\eta\ts (0,\ldots,0) = 1$ together with the
  linear recursion relations
\[
   \eta\ts (K^{}_{1} m^{}_{1} \! +\!\ts r^{}_{1}, \ldots, 
             K^{}_{d}\ts m^{}_{d} \ts\! +\!\ts r^{}_{d} )
    \, =  \sum_{s^{}_{1}=0}^{1} \dots \sum_{s^{}_{d}=0}^{1}
      \begin{pmatrix}  r^{}_{1},\ldots,r^{}_{d} \\ 
         s^{}_{1},\ldots, s^{}_{d} \end{pmatrix} \,
      \eta\ts (m^{}_{1} \! + \!\ts s^{}_{1}, \ldots, 
          m^{}_{d} \ts \! + \!\ts s^{}_{d})
\]
with rational coefficients\/ $\left(\begin{smallmatrix}
    r^{}_{1} ,\ldots, r^{}_{d} \\ s^{}_{1},\ldots, s^{}_{d} 
\end{smallmatrix}\right)$ of modulus\/ $\le 1$. This recursion 
determines all coefficients uniquely once the initial condition\/
$\eta\ts (0,\ldots,0)$ is given.

In particular, one has\/ $\left(\begin{smallmatrix}  0\, ,\ldots, \ts 0 \\
    s^{}_{1},\ldots, s^{}_{d} \end{smallmatrix}\right) =
\delta^{}_{s^{}_{1},0} \cdot \ldots \cdot \delta^{}_{s^{}_{d},0}$, so
that the relation
\[
    \eta\ts (K^{}_{1} m^{}_{1}, \ldots, K^{}_{d}\ts m^{}_{d} )
    \, = \, \eta\ts (m^{}_{1}, \ldots, m^{}_{d})
\]
   holds for all\/ $(m^{}_{1},\ldots,m^{}_{d}) \in \ZZ^{d}$.
\end{lemma}
\begin{proof}
  It is clear that $\eta\ts (0,\ldots,0)=1$. For a general $\eta\ts
  (K^{}_{1} m^{}_{1} \! +\!\ts r^{}_{1}, \ldots, K^{}_{d}\ts m^{}_{d}
  \ts\! +\!\ts r^{}_{d} )$, we start from the right-hand side of
  Eq.~\eqref{eq:etagen} (for finite $N$) and split the summation over
  the $i$th index modulo $K_{i}$.  For each term, we can now apply the
  recursion \eqref{eq:gen-rec}.  The sum can then be regrouped to
  contribute to coefficients of the form $\eta\ts (m^{}_{1}\! +\!
  s^{}_{1}, \ldots, m^{}_{d}\! + \! s^{}_{d})$ with $0\le s^{}_{i}\le 1$.
  In the limit as $N\to\infty$, the contributions sum up to
\begin{equation}\label{eq:sumkap}
   \begin{pmatrix}  r^{}_{1},\ldots,r^{}_{d} \\ 
         s^{}_{1},\ldots, s^{}_{d} \end{pmatrix}
   \, = \,
   \sum_{t^{}_{1}=s^{}_{1}(K^{}_{1}-r^{}_{1})}^{K^{}_{1}-1 
   -r^{}_{1}(1-s^{}_{1})}  \dots
   \sum_{t^{}_{d}=s^{}_{d}(K^{}_{d}-r^{}_{d})}^{K^{}_{d}-1 
   -r^{}_{d}(1-s^{}_{d})} 
   \frac{\kappa^{}_{t^{}_{1},\ldots,t^{}_{d}} \, 
   \kappa^{}_{t^{}_{1}+r^{}_{1}-s^{}_{1}K^{}_{1},\ldots,
   t^{}_{d}+r^{}_{d}-s^{}_{d}K^{}_{d}}}{K^{}_{1}\cdot\ldots\cdot K^{}_{d}}    
\end{equation}
which follows by carefully keeping track of the various overflow rules
due to the calculations modulo $K_{i}$. The claim on the rationality and
boundedness is then obvious. 

Inspecting the equations for the autocorrelation coefficients, one
notices that for some of them (for instance those with $m^{}_{1} =
\ldots = m^{}_{d} = 0$ and one $r_{i} = 1$) the resulting equation is
a linear equation for a single coefficient that occurs on both
sides. The prefactor of this coefficient on the right hand side can
never be $1$, as one can see from Eq.~\eqref{eq:sumkap} together with
the assumption that the $\kappa$-array does contain entries of both
signs. So, this type of equation always has a unique solution in terms
of $\eta\ts (0,\ldots,0)$, which fixes this particular $\eta$
value. The same argument applies to all equations of this type. The
remaining coefficients are then determined recursively, and the
solution is unique as soon as $\eta\ts (0,\ldots,0)$ is specified.

The final claim is clear for $s^{}_{1}=\ldots =s^{}_{d}=0$. 
If $s^{}_{i}=1$ for some $i$, the corresponding sum is empty, and
the coefficient vanishes.
\end{proof}

By construction, the mapping $\eta\! : \, \ZZ^{d} \longrightarrow \RR$
is a positive definite function on $\ZZ^{d}$, with $\eta\ts
(0,\ldots,0)=1$. By the Herglotz-Bochner theorem, see \cite{Rudin} or
Lemma~\ref{lem:gen-Wiener} below, there is then a unique probability
measure $\mu$ on $\TT^{d} = [0,1)^{d}$ such that
\begin{equation}\label{eq:gen-mu}
     \eta\ts (\bs{m}) \, = \,
     \int_{0}^{1} \cdots \int_{0}^{1}
     \exp (2 \pi \ii\ts\ts \bs{mx} ) 
    \dd \mu (\bs{x}) \ts ,
\end{equation}
where we write $\bs{m} = (m^{}_{1}, \ldots , m^{}_{d})$ etc.\ in the
remainder of this section, with $\bs{mx}$ denoting the scalar product.
The positive measure $\mu$ has a unique decomposition into three
components,
\[
    \mu \, = \, \mu^{}_{\mathsf{pp}} + \mu^{}_{\mathsf{sc}} 
      + \mu^{}_{\mathsf{ac}} \ts ,
\]
with Lebesgue measure (as the Haar measure on $\TT^{d}$) as
reference measure.

\begin{prop}\label{prop:pure-type}
  Let $\varrho$ be as in Lemma~$\ref{lem:gen-eta}$, with
  autocorrelation coefficients $\eta$ and associated positive
  measure $\mu$ on $\TT^{d}$. Then, the spectral type of\/ $\mu$
  is pure, which means that it is either a pure point measure,
  a purely singular continuous measure, or a purely 
  absolutely continuous measure.
\end{prop}

\begin{proof}
   By standard arguments, we know that $\mu^{}_{\mathsf{pp}} \perp
   \mu^{}_{\mathsf{sc}} \perp \mu^{}_{\mathsf{ac}} \perp
   \mu^{}_{\mathsf{pp}}$ in the measure-theoretic sense. Now, 
   the recursion relations imply a set of functional relations
   for $\mu$, which must hold for each component separately
   as a consequence of the orthogonality relations. 

   Defining $\eta^{}_{\alpha} (\bs{m}) = \check{\mu}^{}_{\alpha}
   (\bs{m})$ for $\alpha \in \{ \mathsf{pp}, \mathsf{sc}, \mathsf{ac}
   \}$ separately, 
   each type of autocorrelation coefficient must then
   satisfy the same recursion relations as $\eta$ itself, subject to
   the condition
\[
   \eta\ts (\bs{0}) \, = \, \eta^{}_{\mathsf{pp}}(\bs{0}) +
   \eta^{}_{\mathsf{sc}} (\bs{0}) + \eta^{}_{\mathsf{ac}} (\bs{0}) 
   \, = \, 1 \ts .
\]
Note that, due to the linearity of the recursion relations, this
splitting indeed leads to a solution of the recursion with the correct
initial condition, hence the unique solution stated in
Lemma~\ref{lem:gen-eta}.

This observation means that the three components of $\eta\ts (\bs{m})$
are proportional according to their individual initial conditions
$\eta^{}_{\alpha} (\bs{0})$, so cannot produce different spectral
type.  This is only compatible with the claim made, and the type is
pure.
\end{proof}

Let us take a closer look at the spectral types. Assume that the
measure $\mu$ is absolutely continuous. Then, the Riemann-Lebesgue
Lemma tells us that $\eta\ts (\bs{m}) = \eta^{}_{\mathsf{ac}} (\bs{m})
\longrightarrow 0$ as $\lvert \bs{m} \rvert \to \infty$. Recalling the
last claim of Lemma~\ref{lem:gen-eta}, it follows that we must have
$\eta\ts (\bs{m}) =0$ for all $\bs{m}\ne \bs{0}$, and hence $\eta\ts
(\bs{m}) = \delta^{}_{\bs{m},\bs{0}}$.  This would imply $\mu$ to be
Lebesgue measure on $\TT^{d}$.  A careful inspection of the recursion
reveals that this outcome is only possible if all coefficients of type
$\left(
\begin{smallmatrix} r^{}_{1},\ldots ,r^{}_{d} \\
0\, , \ldots , \, 0 \end{smallmatrix}\right)$ vanish. 
However, we have
\[
   \left(\begin{array}{@{}c@{}c@{}c@{}}  
   K^{}_{1}\!-\! 1 &,\ldots, &K^{}_{d}\! -\! 1 \\ 
         0& ,\ldots, &0 \end{array}\right)
   \, = \,
   \frac{\kappa^{}_{0,\ldots,0}\, 
   \kappa^{}_{K^{}_{1}-1,\ldots,K^{}_{d}-1}}
   {K^{}_{1}\cdot\ldots\cdot K^{}_{d}} \, \ne\, 0
\]
which excludes this possibility. Consequently, the measure
$\mu$ in Proposition~\ref{prop:pure-type} is a singular measure.

Put together, we have proved the following result.
\begin{theorem}\label{thm:spec-type}
  Let\/ $\varrho$ be a primitive, bijective block substitution on\/
  $\ZZ^{d}$, for a binary alphabet. Assume that it is extensive in all
  directions, and let $w$ be a fixed point configuration that
  satisfies Eq.~\eqref{eq:gen-rec} on\/ $\ZZ^{d}$.

  If\/ $\eta$ is the corresponding autocorrelation function on\/
  $\ZZ^{d}$, and $\mu$ the attached positive measure on $\TT^{d}$
  according to Eq.~\eqref{eq:gen-mu}, the measure\/ $\mu$ is 
  singular. In particular, it is either 
  pure point or purely singular continuous.   \qed
\end{theorem}

Both possibilities clearly occur. Examples for pure point cases can be
constructed via substitutions that have fully periodic fixed
points. Such fixed points usually have non-trivial height lattices,
and are thus of limited interest; see \cite[Sec.~3.1]{Nat1} for the
precise definition of the height lattice of a $\ZZ^{d}$ lattice
substitution.  If the height lattice of a primitive bijective
substitution is trivial, one can construct a cyclic space in the
complement of the span of the eigenfunctions \cite{Nat1,Nat2}.  Its
spectral measure is nothing but our measure $\mu$, which is then
inevitably singular continuous.

\begin{coro}
  If the block substitution from Theorem~$\ref{thm:spec-type}$ has
  trivial height, the diffraction measure of the corresponding
  balanced Dirac comb is purely singular continuous.  \qed
\end{coro}

\begin{remark}\label{rem:strange}
  Singular continuity of a (positive) measure $\mu$ comprises rather
  different situations in $d>1$ dimensions.  For instance, $\mu$ can
  be genuinely singular continuous, such as the measure of our squiral
  tiling above. However, one can also find product measures of a
  different kind. As an example, consider the primitive, bijective block
  substitution defined by
\[
   \begin{array}{ccccc}
     & & \bar{1} & 1 & \bar{1} \\
   1 &  \longmapsto & 1 & \bar{1} & 1
   \end{array}
\]
in symbolic notation on the binary alphabet $\{1,\bar{1}\}$.
This rule defines a hull that is $2$-periodic in the horizontal
direction and Thue-Morse in the vertical one. Consequently, the
measure $\mu$ in this case is a product measure, written as
$\mu = \mu^{}_{1} \otimes \mu^{}_{2}$, where $\mu^{}_{1}$ is a
pure point measure and $\mu^{}_{2}$ is the classic TM measure.
Nevertheless, $\mu$ is purely singular continuous as a measure
on $\TT^{2}$. This is an example with a non-trivial height lattice
that still shows purely singular continuous diffraction, though
the corresponding distribution function is \emph{not}
continuous.
\end{remark}

Let us assume that the probability measure $\mu$ from
Eq.~\eqref{eq:gen-mu} is purely singular continuous and that the
distribution function $F$ for $\widehat{\gamma} = \mu *
\delta^{}_{\ZZ^{d}}$ is continuous (which excludes examples such as
those mentioned in Remark~\ref{rem:strange}). The latter condition is
equivalent to the $d$ one-dimensional distribution functions
$F(1,\ldots,1,x_{i},1,\ldots,1)$ being continuous, by an application
of \cite[Thm.~2.3]{Tucker}. Then, the recursion relation for the
autocorrelation coefficients $\eta$ can meaningfully be turned into an
iteration, as in our planar example. One can then derive an explicit
representation of $\mu$ (and hence also $\widehat{\gamma}$) as a
(generalised) Riesz product in this case. In particular, iterating the
corresponding functional equation once with initial condition $F^{(0)}
(\bs{x}) = \prod_{i=1}^{d} x_{i}$ leads to $F^{(1)}$, with density
function
\[
    f^{(1)} (\bs{x}) \, = \,
    \frac{\partial^{d}}{\partial x^{}_{1} \dots \partial x^{}_{d}}
    F^{(1)} (\bs{x}) \ts , 
\]
which plays the role of the trigonometric function $\vartheta$ in
this more general setting. It can
directly be calculated by means of the recursion coefficients as
\begin{equation}\label{eq:gen-theta}
   \vartheta (\bs{x}) \, = 
   \sum_{r^{}_{1} = 1-K^{}_{1}}^{K^{}_{1} - 1} \dots 
   \sum_{r^{}_{d} = 1 - K^{}_{d}}^{K^{}_{d} - 1}
   \begin{pmatrix}  \lvert r^{}_{1} \rvert , 
         \ldots , \lvert r^{}_{d} \rvert \\ 
         0 \;\, , \ldots , \;\, 0 \end{pmatrix}
   \prod_{i=1}^{d} \cos (2 \pi r^{}_{i} \ts x^{}_{i})\ts ,
\end{equation}
which defines a non-negative function on $\RR^{d}$ with
$\int_{\TT^{d}} \vartheta (\bs{x}) \dd \bs{x} = 1$.  This leads to
$f^{(N)} (\bs{x}) = \prod_{\ell = 0}^{N-1} \vartheta (K_{1}^{\ell}
x^{}_{1}, \ldots , K_{d}^{\ell} x^{}_{d} )$ and thus to an infinite
Riesz product for the diffraction measure $\widehat{\gamma}$.
The distribution function can also be written as
\[
   F (\bs{x}) \, = \, x^{}_{1} \cdot \ldots \cdot x^{}_{d}
   \sum_{\bs{m} \in \ZZ^{d}} \eta\ts (\bs{m})
   \prod_{i=1}^{d} \mathrm{sinc} (2 \pi m^{}_{i} x^{}_{i} ) \ts ,
\]
in complete analogy to our previous formulas of this type.
\smallskip

A complete classification of the cases with pure point spectrum and an
extension to larger alphabets remain as interesting open problems.
Another question concerns a better understanding of the topological
factors of bijective substitutions. We expect a complete
correspondence between the original dynamical spectrum and the
diffraction spectra of the system and its factors.

\section*{Acknowledgements}

We are grateful to Natalie Priebe Frank, Franz G\"ahler, Tilmann
Gneiting, Holger K\"osters, Daniel Lenz and Robbie Robinson for
helpful discussions. This work was supported by the German Research
Council (DFG), within the CRC 701.

\appendix

\section*{Appendix}

Let us formulate Wiener's lemma for the group $\ZZ^{d}$, which is a
straight-forward generalisation of the case $d=1$ in
\cite[Sec.~4.16]{Nad}. Here, $C_{n} = \{ u \in \RR^{d} \mid \| u
\|^{}_{\infty} \le n \}$ denotes the closed cube of sidelength $2n$,
centred at the origin. One has $\mathrm{card} (\ZZ^{d}\cap C_{n}) =
(2n\! + \!  1)^{d}$.
\begin{lemma}\label{lem:gen-Wiener}
  If\/ $\eta \! : \, \ZZ^{d} \longrightarrow \CC$ is positive
  definite, there is a unique positive measure\/ $\mu$ on\/
  $\TT^{d}\simeq \RR^{d}/\ZZ^{d}$ such that
\[
   \eta\ts(k) \, =  \int_{\TT^{d}} \ee^{2 \pi \ii k u} \dd \mu (u)
\]
   holds for all\/ $k\in\ZZ^{d}$. When\/ $\eta\ts (0) = 1$, the
   measure\/ $\mu$ is a probability measure on\/ $\TT^{d}$.

   Moreover, one has the relation
\[
   \lim_{n\to\infty} \frac{1}{(2n+1)^{d}}
   \sum_{k \in \ZZ^{d}\cap C_{n}} \bigl| \eta\ts (k) \bigr|^{2}
   \, = \, \sum_{t\in \TT^{d}} \bigl( \mu (\{t\}) \bigr)^{2} ,
\]
where the last sum runs over at most countably many points.
\end{lemma}
\begin{proof}
  Since $\ZZ^{d}$ and $\TT^{d}$ form a mutually dual pair of locally
  compact Abelian groups, the first claim is just the Herglotz-Bochner
  theorem for this situation, compare \cite{Rudin}, with
  $\mu (\TT^{d}) = \eta\ts (0)$.

  Now, define $g_{n} (u) = \frac{1}{(2n+1)^{d}} \sum_{k\in\ZZ^{d}\cap
    C_{n}} \ee^{2 \pi \ii k u}$, which is bounded by $1$. It is not
    difficult to see that $g_{n} (u) \xrightarrow{\, n\to\infty\,}
    1_{\{0\}} (u)$ holds for any $u\in\TT^{d} = [0,1)^{d}$. Using
    Fubini's theorem and noting that $\bar{\mu} = \mu$, one obtains
\[
   \begin{split}
   \frac{1}{(2n\! +\! 1)^{d}} &
   \sum_{k \in \ZZ^{d}\cap C_{n}} \!\! \bigl| \eta\ts (k) \bigr|^{2}
   \, = \,  \frac{1}{(2n\! + \!1)^{d}} \sum_{k \in \ZZ^{d}\cap C_{n}}
   \int_{\TT^{d}}\! \ee^{2\pi\ii ku} \dd\mu (u)
   \overline{\int_{\TT^{d}}\! \ee^{2\pi\ii kv} \dd\mu (v)} \\
   & \hspace*{-10mm} = \int_{\TT^{d} \times \TT^{d}} g_{n} (u-v) 
   \dd\mu (u) \dd \mu (v)
   \;\, \xrightarrow{\, n\to\infty\,}\, \int_{\TT^{d} \times \TT^{d}}
   1_{\{0\}} (u-v) \dd\mu (u) \dd \mu (v)
   \end{split}
\]
where the last step follows from dominated convergence; compare
\cite[Thm.~6.1.15]{Ped} for a formulation in sufficient generality.
The integral over the diagonal evaluates via Fubini's
theorem as
\[
   \int_{\TT^{d}} \int_{\TT^{d}}
   1_{\{0\}} (u-v) \dd\mu (u) \dd \mu (v)
   \, =  \int_{\TT^{d}} \mu (\{v\}) \dd \mu (v)
   \, =  \sum_{t\in\TT^{d}} \bigl( \mu (\{ t \}) \bigr)^{2} ,
\]
where the last sum (since $\mu$ is a positive measure) runs over all
points $t\in\TT^{d}$ with $\mu (\{t\}) > 0$, which are (at most)
countably many.
\end{proof}

\end{document}